\documentclass[11pt]{article}
\usepackage{amsthm,geometry,amssymb,amsmath,enumerate,float,tikz,cite,algorithm2e}
\tikzstyle{every node}=[circle, draw, inner sep=0pt, minimum width=4pt]

\newtheorem{theorem}{Theorem}[section]
\newtheorem{lemma}{Lemma}[section]

\title{On the equality of the induced matching number and the \\
uniquely restricted matching number for subcubic graphs}
\author{M. F\"{u}rst \and D. Rautenbach}
\date{}

\begin{document}

\maketitle

\begin{center}
{\small 
Institute of Optimization and Operations Research, Ulm University, Ulm, Germany\\
\texttt{$\{$maximilian.fuerst,dieter.rautenbach$\}$@uni-ulm.de}\\[3mm]
}
\end{center}

\begin{abstract}
For a matching $M$ in a graph $G$,
let $G(M)$ be the subgraph of $G$ 
induced by the vertices of $G$ 
that are incident with an edge in $M$.
The matching $M$ is induced,
if $G(M)$ is $1$-regular,
and $M$ is uniquely restricted, 
if $M$ is the unique perfect matching of $G(M)$.
The induced matching number $\nu_s(G)$ of $G$ 
is the largest size of an induced matching in $G$,
and
the uniquely restricted matching number $\nu_{ur}(G)$ of $G$ 
is the largest size of a uniquely restricted matching in $G$.

Golumbic, Hirst, and Lewenstein 
(Uniquely restricted matchings, Algorithmica 31 (2001) 139-154)
posed the problem to characterize the graphs $G$ with $\nu_s(G)=\nu_{ur}(G)$.
We give a complete characterization of the 
$2$-connected subcubic graphs $G$ of sufficiently large order 
with $\nu_s(G)=\nu_{ur}(G)$.
As a consequence,
we are able to show that the subcubic graphs $G$ with $\nu_s(G)=\nu_{ur}(G)$
can be recognized in polynomial time.
\end{abstract}
{\small 
\begin{tabular}{lp{13cm}}
\end{tabular}
}

\pagebreak

\section{Introduction}
We consider only simple, finite, and undirected graphs, and use standard terminology.
For a graph $G$, and a matching $M$ in $G$, let $V(M)$ be the set of vertices 
covered by $M$, and let $G(M)$ be the subgraph of $G$ induced by $V(M)$. 
A matching $M$ in $G$ 
is \textit{induced} \cite{ca1} 
or \textit{uniquely restricted} \cite{gohile}
if $G(M)$ is $1$-regular
or $M$ is the unique perfect matching of $G(M)$, respectively.
Let $\nu(G)$, $\nu_s(G)$, and $\nu_{ur}(G)$
be the maximum sizes of 
an ordinary, 
an induced, and
a uniquely restricted matching, respectively.
Golumbic, Hirst, and Lewenstein \cite{gohile} 
observed that a matching $M$ in a graph $G$ is
uniquely restricted 
if and only if
there is no $M$-alternating cycle in $G$.
Since every induced matching is uniquely restricted,
$$\nu_s(G)\leq  \nu_{ur}(G)  \leq \nu(G)$$
for every graph $G$.
Induced matchings are also known as {\it strong} matchings.

The computational hardness of finding 
maximum induced matchings or 
maximum uniquely restricted matchings 
in a given graph has been shown in \cite{ca1, gohile, stva}.
Nevertheless, 
it can be decided in polynomial time,
whether a given graph $G$ 
satisfies $\nu(G) = \nu_s(G)$ \cite{cawa, jora, koro}
or $\nu(G) = \nu_{ur}(G)$ \cite{lema, peraso}.
In \cite{gohile}
Golumbic, Hirst, and Lewenstein 
pose the problem to characterize the graphs $G$ with $\nu_s(G)=\nu_{ur}(G)$.

As our first main result, 
we give a complete characterization 
of the $2$-connected subcubic graphs $G$ 
of sufficiently large order with $\nu_s(G)=\nu_{ur}(G)$.
Apart from some small sporadic graphs, 
all these graphs have a rather simple structure.

See Figure \ref{fig1} for illustrations of the following graphs.
\begin{itemize}
\item For a positive integer $k$,
let $L_k$ be the graph of order $3k$ 
that arises from $k$ vertices $w_1,\ldots,w_k$, 
and two disjoint paths $u_1u_2\ldots u_k$ and $v_1v_2\ldots v_k$,
by adding the edges $w_iu_i$ and $w_iv_i$ for every $i \in \lbrack k \rbrack$,
where $[k]$ denotes the set of positive integers at most $k$.
\item Let $L_k'$ arise from $L_k$ by adding the two new vertices $w_1'$ and $w_k'$,
and the six new edges 
$u_1w_1'$,
$v_1w_1'$,
$u_1v_1$,
$u_kw_k'$,
$v_kw_k'$, and
$u_kv_k$.
\item Let $\mathcal{B}_1$ be the set of all $2$-connected subcubic graphs $G$
such that there is some positive integer $k$
for which $L_k$ is an induced subgraph of $G$,
and $G$ is a subgraph of $L_k'$.
(Note that there are six non-isomorphic choices for such a graph $G$
with $L_k\subseteq G\subseteq L_k'$.)
\item Let $\mathcal{B}_2$ 
be the set of all subcubic graphs $G$
such that there is some positive integer $k$ at least $3$
for which $G$ arises from a $L_k$ by
\begin{itemize}
\item adding the two new edges $u_1v_k$ and $v_1u_k$, if $k$ is odd,
\item adding the two new edges $u_1u_k$ and $v_1v_k$, if $k$ is even.
\end{itemize}
\item Let $\mathcal{B}=\mathcal{B}_1\cup \mathcal{B}_2$.
\end{itemize}
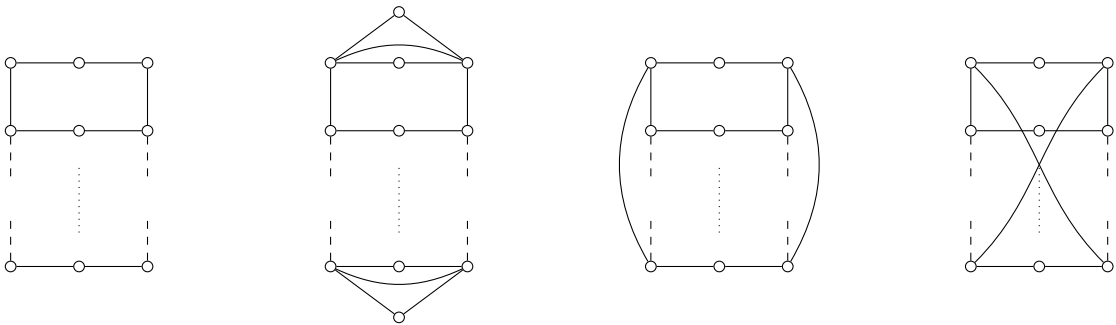
\begin{figure}[H]
\begin{minipage}{0.24\textwidth}
\centering\tiny
\begin{tikzpicture}[scale=0.9]
	    \node (v1) at (-1,4) {};
	    \node (v2) at (0,4) {};
	    \node (v3) at (1,4) {};
	    \node (w1) at (-1,3) {};
	    \node (w2) at (0,3) {};
	    \node (w3) at (1,3) {};
	    \node (x1) at (-1,1) {};
	    \node (x2) at (0,1) {};
	    \node (x3) at (1,1) {};
	    
	    \foreach \from/\to in {v1/v2, v2/v3, v1/w1, v3/w3, w1/w2, w2/w3, x1/x2, x2/x3}
	    \draw [-] (\from) -- (\to);

	    \draw[-,dotted] (0,1.5) -- (0,2.5);
	    \draw[-,dashed] (w1) -- (-1,2.25);
	    \draw[-,dashed] (w3) -- (1,2.25);
	    \draw[-,dashed] (x1) -- (-1,1.75);
	    \draw[-,dashed] (x3) -- (1,1.75);

\end{tikzpicture}
\end{minipage}
\begin{minipage}{0.24\textwidth}
\centering\tiny
\begin{tikzpicture}[scale = 0.9] 
	    \node (u) at (0,4.75) {};
	    \node (v1) at (-1,4) {};
	    \node (v2) at (0,4) {};
	    \node (v3) at (1,4) {};
	    \node (w1) at (-1,3) {};
	    \node (w2) at (0,3) {};
	    \node (w3) at (1,3) {};
	    \node (x1) at (-1,1) {};
	    \node (x2) at (0,1) {};
	    \node (x3) at (1,1) {};
	    \node (y) at (0,0.25) {};
	    \foreach \from/\to in {u/v1, u/v3, v1/v2, v2/v3, v1/w1, v3/w3, w1/w2, w2/w3, x1/x2, x2/x3}
	    \draw [-] (\from) -- (\to);
	    
	    \foreach \from/\to in {y/x1, y/x3}
	    \draw [-] (\from) -- (\to);
	    
	    \draw[-,] (x1) to[out=-25,in=-155] (x3);
	    \draw[-,] (v1) to[out=25,in=155] (v3);
	    
	    \draw[-,dotted] (0,1.5) -- (0,2.5);
	    \draw[-,dashed] (w1) -- (-1,2.25);
	    \draw[-,dashed] (w3) -- (1,2.25);
	    \draw[-,dashed] (x1) -- (-1,1.75);
	    \draw[-,dashed] (x3) -- (1,1.75);
	    
\end{tikzpicture}
\end{minipage}
\begin{minipage}{0.24\textwidth}
\centering\tiny
\begin{tikzpicture}[scale = 0.9] 
	    \node (v1) at (-1,4) {};
	    \node (v2) at (0,4) {};
	    \node (v3) at (1,4) {};
	    \node (w1) at (-1,3) {};
	    \node (w2) at (0,3) {};
	    \node (w3) at (1,3) {};
	    \node (x1) at (-1,1) {};
	    \node (x2) at (0,1) {};
	    \node (x3) at (1,1) {};
	    
	    \foreach \from/\to in {v1/v2, v2/v3, v1/w1, v3/w3, w1/w2, w2/w3, x1/x2, x2/x3}
	    \draw [-] (\from) -- (\to);
	    
	    \draw[-] (v3) to[out=-60,in=60] (x3);
	    \draw[-] (v1) to[out=-120,in=120] (x1);
	    
	    \draw[-,dotted] (0,1.5) -- (0,2.5);
	    \draw[-,dashed] (w1) -- (-1,2.25);
	    \draw[-,dashed] (w3) -- (1,2.25);
	    \draw[-,dashed] (x1) -- (-1,1.75);
	    \draw[-,dashed] (x3) -- (1,1.75);
\end{tikzpicture}
\end{minipage}
\begin{minipage}{0.24\textwidth}
\centering\tiny
\begin{tikzpicture}[scale = 0.9] 
	    \node (v1) at (-1,4) {};
	    \node (v2) at (0,4) {};
	    \node (v3) at (1,4) {};
	    \node (w1) at (-1,3) {};
	    \node (w2) at (0,3) {};
	    \node (w3) at (1,3) {};
	    \node (x1) at (-1,1) {};
	    \node (x2) at (0,1) {};
	    \node (x3) at (1,1) {};
	    
	    \foreach \from/\to in {v1/v2, v2/v3, v1/w1, v3/w3, w1/w2, w2/w3, x1/x2, x2/x3}
	    \draw [-] (\from) -- (\to);
	    
	    \draw[-] (v1) to[out=-45,in=135] (x3);
	    \draw[-] (v3) to[out=-135,in=45] (x1);
	    
	    \draw[-,dotted] (0,1.5) -- (0,2.5);
	    \draw[-,dashed] (w1) -- (-1,2.25);
	    \draw[-,dashed] (w3) -- (1,2.25);
	    \draw[-,dashed] (x1) -- (-1,1.75);
	    \draw[-,dashed] (x3) -- (1,1.75);
\end{tikzpicture}
\end{minipage}
\caption{An illustration of 
$L_k$,
$L_k'$,
a graph in $\mathcal{B}_2$ for even $k$, and 
a graph in $\mathcal{B}_2$ for odd $k$.} \label{fig1}
\end{figure}
Here is our first main result.

\begin{theorem} \label{th_char}
If $G$ is a $2$-connected subcubic graph of order at least $21$,
then $\nu_s(G) = \nu_{ur}(G)$ 
if and only if $G$ is isomorphic to a graph in 
$\mathcal{B}$. 
\end{theorem}
The reason for the assumption that $G$ has order at least $21$ 
is that there are several small $2$-connected subcubic graphs $G$ 
that satisfy $\nu_s(G) = \nu_{ur}(G)$ 
but are not isomorphic to a graph in $\mathcal{B}$. 
While the proof of Theorem \ref{th_char}
relies on very simple observations
captured by Lemma \ref{lemma0},
it involves a rather detailed case analysis.
Theorem \ref{th_char} is a key ingredient of our second main result.

\begin{theorem} \label{t3}
Deciding whether a given subcubic graph $G$ 
satisfies $\nu_{ur}(G) = \nu_s(G)$ can be done in polynomial time.
\end{theorem}
In the next section, we first prove Theorem \ref{t3},
and in Section \ref{sec3}, we prove Theorem \ref{th_char}. 

\section{Deciding $\nu_s(G)=\nu_{ur}(G)$ for a given subcubic graph}\label{sec2}

Our first lemma collects properties of general graphs satisfying the considered equality.

\begin{lemma}\label{lemma0}
Let $G$ be a graph with $\nu_s(G)=\nu_{ur}(G)$, 
and let $M$ be some maximum induced matching in $G$.
\begin{enumerate}[(i)]
\item If $uv$ is an edge of $G-V(M)$, then 
there is a $4$-cycle $uvwxu$ with $wx \in M$.
\item If $u_1u_2$ and $v_1v_2$ are two disjoint edges 
such that $u_1v_1 \in M$, then
\begin{itemize}
\item either $u_2$ and $v_2$ are adjacent,
\item or $u_1$ is adjacent to $v_2$, and $v_1$ is adjacent to $u_2$,
\item or there is a $6$-cycle $u_1u_2uvv_2v_1u_1$, where $uv \in M$.
\end{itemize}
\end{enumerate}
\end{lemma}
\begin{proof}
(i) If the stated $4$-cycle does not exist, 
then $M\cup \{ uv\}$ is a uniquely restricted matching in $G$ 
with more than $|M|=\nu_s(G)=\nu_{ur}(G)$ edges, 
which is a contradiction.

(ii) If none of the three situations arises,
then $(M \cup \{u_1u_2,v_1v_2 \}) \setminus \{u_1v_1 \}$ 
is a uniquely restricted matching in $G$
with more than $|M|=\nu_s(G)=\nu_{ur}(G)$ edges, 
which is a contradiction.
\end{proof}
An immediate consequence of Lemma \ref{lemma0} is the following.

\begin{lemma}\label{lemma1}
If $G$ is a connected subcubic graph with $\nu_s(G)=\nu_{ur}(G)$ 
and minimum degree at least $2$, 
then $G$ is $2$-connected.
\end{lemma}
\begin{proof} 
Suppose, for a contradiction, that $G$ is not $2$-connected.
Since $G$ is subcubic, this implies that $G$ has a bridge $uv$.
Let $M$ be some maximum induced matching in $G$.
By Lemma \ref{lemma0}(i), every edge of $G-V(M)$ lies in a $4$-cycle,
that is, $uv$ is not an edge of $G-V(M)$.

If $u,v\in V(M)$, then $uv\in M$. 
Since $G$ has minimum degree at least $2$,
$u$ has a neighbor $u'$ distinct from $v$, and 
$v$ has a neighbor $v'$ distinct from $u$.
Since $uv$ is a bridge, the edges $uu'$ and $vv'$ are disjoint.
Now, Lemma \ref{lemma0}(ii) implies the contradiction
that $uv$ lies in a cycle of length at most $6$.
Hence, we may assume that $u\in V(G)\setminus V(M)$ and $v\in V(M)$.

Let $M$ contain the edge $vv'$.
Since $G$ has minimum degree at least $2$,
$v'$ has a neighbor $v''$ in $V(G)\setminus V(M)$.
Since $uv$ is a bridge, the edges $uv$ and $v'v''$ are disjoint.
Now, Lemma \ref{lemma0}(ii) implies the contradiction
that $uv$ lies in a cycle of length $4$ or $6$.
This completes the proof.
\end{proof}
Let ${\cal B}'$ be the set of all graphs $G$ such that 
\begin{itemize}
\item either $G$ has order at most $20$, and satisfies $\nu_s(G)=\nu_{ur}(G)$,
\item or $G$ has order at least $21$, and is isomorphic to a graph in ${\cal B}$.
\end{itemize}
In order to prove Theorem \ref{t3},
we consider two algorithms.

\medskip

\begin{algorithm}[H]
\LinesNumbered\SetAlgoLined
\KwIn{A subcubic graph $G$.}
\KwOut{Either a maximum uniquely restricted matching $M$ in $G$,
or the correct statement ``$\nu_s(G)\not=\nu_{ur}(G)$''.}
\Begin{
\LinesNumbered\SetAlgoLined
$M\leftarrow \emptyset$\;
$H\leftarrow G$\;
\While{$H$ has an edge $uv$ with $d_H(u)=1$\label{l3}}
{
$M\leftarrow M\cup \{ uv\}$\;
$H\leftarrow H-\{ u,v\}$\;
}\label{l6}
Let $H_1,\ldots,H_k$ be the connected components of $H$\; \label{l8}
\If{$H_i\not\in {\cal B}'$ for some $i\in [k]$\label{l9}}
{
\Return{``$\nu_s(G)\not=\nu_{ur}(G)$''}\;
{\bf break}\;
}\label{l11}
Let $M_i$ be a maximum uniquely restricted matching in $H_i$ for $i\in [k]$\label{l12}\;
\Return{$M\cup M_1\cup\cdots\cup M_k$}\label{l13}\;
}
\caption{{\sc Murm}}\label{alg1}
\end{algorithm}
The correctness of this algorithms relies on the following lemma.

\begin{lemma}\label{lemma2}
Let $G$ be a graph, 
and let $uv$ be an edge of $G$ with $d_G(u)=1$.

Let $G'=G-\{ u,v\}$, 
and let $M'$ be a maximum uniquely restricted matching in $G'$.
\begin{enumerate}[(i)]
\item $\{ uv\}\cup M'$ is a maximum uniquely restricted matching in $G$.
\item If $\nu_s(G)=\nu_{ur}(G)$, then 
$\nu_s(G')=\nu_{ur}(G')$.
\end{enumerate}
\end{lemma}
\begin{proof} 
(i) follows immediately from the simple observation that some
maximum uniquely restricted matching in $G$ contains $uv$.
If $\nu_s(G)=\nu_{ur}(G)$, then
$$\nu_s(G)\leq \nu_s(G')+1\leq \nu_{ur}(G')+1
\stackrel{(i)}{=}\nu_{ur}(G)=\nu_s(G)$$
implies $\nu_s(G')=\nu_{ur}(G')$,
and (ii) follows.
\end{proof} 

\begin{lemma}\label{lemma3}
Algorithm \ref{alg1} {\sc Murm} works correctly,
and can be implemented to run in polynomial time.
\end{lemma}
\begin{proof} 
By Lemma \ref{lemma2}(i),
the set $M$ constructed by the {\bf while}-loop
in lines \ref{l3} to \ref{l6} is a subset of some 
maximum uniquely restricted matching in $G$.

If $\nu_s(G)=\nu_{ur}(G)$,
then, by Lemma \ref{lemma2}(ii),
the graph $H$ in line \ref{l8} satisfies 
$\nu_s(H)=\nu_{ur}(H)$.
Since the graph $H$ in line \ref{l8} has no vertex of degree $1$,
Lemma \ref{lemma1} and Theorem \ref{th_char} imply
that every component of $H$ belongs to ${\cal B}'$.
Therefore, if some component of $H$ 
does not belong to ${\cal B}'$,
then {\sc Murm} correctly returns 
``$\nu_s(G)\not=\nu_{ur}(G)$''.
In view of the simple structure of the graphs in ${\cal B}$,
it can be decided in polynomial time 
whether a given graph belongs to ${\cal B}'$,
that is, the {\bf if}-statement in lines \ref{l9} to \ref{l11}
can be implemented to run in polynomial time.

Now, we may assume that 
every component of $H$ belongs to ${\cal B}'$.
By Lemma \ref{lemma2}(i),
the matching returned in line \ref{l13}
is a maximum uniquely restricted matching in $G$.
Furthermore, 
again in view of the simple structure of the graphs in ${\cal B}$,
a maximum uniquely restricted matching 
can be determined in polynomial time 
for every given graph in ${\cal B}'$,
that is, line \ref{l12} 
can be implemented to run in polynomial time.
\end{proof}

\begin{algorithm}[H]
\LinesNumbered\SetAlgoLined
\KwIn{A subcubic graph $G$.}
\KwOut{Either an induced matching $M$ in $G$,
or the correct statement ``$\nu_s(G)\not=\nu_{ur}(G)$''.}
\Begin{
\LinesNumbered\SetAlgoLined
$M\leftarrow \emptyset$\;
$H\leftarrow G$\;
\While{$H$ has an edge $uv$ with $d_H(u)=1$\label{l3b}}
{
$M\leftarrow M\cup \{ uv\}$\;
$H\leftarrow H-N_H[v]$\;
}\label{l6b}
Let $H_1,\ldots,H_k$ be the connected components of $H$\; \label{l8b}
\If{$H_i\not\in {\cal B}'$ for some $i\in [k]$\label{l9b}}
{
\Return{``$\nu_s(G)\not=\nu_{ur}(G)$''}\;
{\bf break}\;
}\label{l11b}
Let $M_i$ be a maximum induced matching in $H_i$ for $i\in [k]$\label{l12b}\;
\Return{$M\cup M_1\cup\cdots\cup M_k$}\label{l13b}\;
}
\caption{{\sc Msm}}\label{alg2}
\end{algorithm}

The correctness of this algorithms relies on the following lemma.

\begin{lemma}\label{lemma4}
Let $G$ be a graph, 
and let $uv$ be an edge of $G$ with $d_G(u)=1$.

Let $G''=G-N_G[v]$, 
and let $M''$ be an induced matching in $G''$.
\begin{enumerate}[(i)]
\item $\{ uv\}\cup M''$ is an induced matching in $G$.
\item If $\nu_s(G)=\nu_{ur}(G)$, then 
$\nu_s(G'')=\nu_s(G)-1=\nu_{ur}(G'')$.
\item If $\nu_s(G)=\nu_{ur}(G)$, and $M''$ 
is a maximum induced matching in $G''$,
then $\{ uv\}\cup M''$ is a maximum induced matching in $G$.
\end{enumerate}
\end{lemma}
\begin{proof} 
(i) is trivial; note that none of the involved matchings is supposed to be maximum.
Now, let $\nu_s(G)=\nu_{ur}(G)$,
and let $M$ be a maximum induced matching in $G$.
If $M$ contains no edge incident with $v$, 
then adding $uv$ to $M$ results in a larger 
uniquely restricted matching in $G$, 
which contradicts $\nu_s(G)=\nu_{ur}(G)$.
Hence, $M$ contains an edge $e$ incident with $v$.
Since $M\setminus \{ e\}$ is an induced matching in $G''$, 
and adding $uv$ to a uniquely restricted matching in $G''$
yields a uniquely restricted matching in $G$,
we obtain 
$$\nu_s(G)\leq \nu_s(G'')+1\leq \nu_{ur}(G'')+1
\leq \nu_{ur}(G)=\nu_s(G)$$
implies $\nu_s(G'')=\nu_{ur}(G'')$,
and (ii) follows.
(iii) follows immediately from (i) and (ii).
\end{proof} 

\begin{lemma}\label{lemma5}
Algorithm \ref{alg2} {\sc Msm} works correctly,
and can be implemented to run in polynomial time.
Furthermore, 
if the input graph $G$ satisfies $\nu_s(G)=\nu_{ur}(G)$,
then {\sc Msm} returns a maximum induced matching in $G$.
\end{lemma}
\begin{proof} 
By Lemma \ref{lemma4}(i),
if {\sc Msm} returns a matching $M$ in line \ref{l13b},
then $M$ is an induced matching in $G$.
If $\nu_s(G)=\nu_{ur}(G)$,
then, by Lemma \ref{lemma4}(ii),
the graph $H$ in line \ref{l8b} satisfies 
$\nu_s(H)=\nu_{ur}(H)$.
Since the graph $H$ in line \ref{l8b} has no vertex of degree $1$,
Lemma \ref{lemma1} and Theorem \ref{th_char} imply
that every component of $H$ belongs to ${\cal B}'$.
Therefore, if some component of $H$ 
does not belong to ${\cal B}'$,
then {\sc Msm} correctly returns 
``$\nu_s(G)\not=\nu_{ur}(G)$''.

Now, let $\nu_s(G)=\nu_{ur}(G)$.
By Lemma \ref{lemma4}(ii) and (iii),
the set $M$ constructed by the {\bf while}-loop
in lines \ref{l3b} to \ref{l6b} is a subset of some 
maximum induced matching in $G$,
which implies that the matching returned in line \ref{l13b}
is a maximum induced matching in $G$.

The statement about the running time follows 
similarly as in the proof of Lemma \ref{lemma3}.
\end{proof} 
It is now easy to complete the following.

\begin{proof} [Proof of Theorem \ref{t3}]
Let $G$ be a given subcubic graph.
We execute 
Algorithm \ref{alg1} {\sc Murm} 
and 
Algorithm \ref{alg2} {\sc Msm} 
on $G$.
If one of the two algorithms returns
the statement ``$\nu_s(G)\not=\nu_{ur}(G)$'',
then this is correct 
by Lemma \ref{lemma3} and Lemma \ref{lemma5}.
Hence, we may assume that 
Algorithm \ref{alg1} {\sc Murm} 
returns a maximum uniquely restricted matching $M_{ur}$ in $G$,
and 
Algorithm \ref{alg2} {\sc Msm} 
returns an induced matching $M_s$ in $G$.

If $\nu_s(G)=\nu_{ur}(G)$,
then, by Lemma \ref{lemma5},
$M_s$ is a maximum induced matching in $G$,
and, hence, $|M_{ur}|=|M_s|$.
Conversely, if $|M_{ur}|=|M_s|$,
then, by Lemma \ref{lemma3} and $\nu_s(G)\leq \nu_{ur}(G)$,
$M_s$ is a maximum induced matching in $G$, and 
$\nu_s(G)=\nu_{ur}(G)$.
Altogether, it follows that 
$\nu_s(G)=\nu_{ur}(G)$ holds if and only if 
$|M_{ur}|=|M_s|$.
\end{proof} 

\section{The $2$-connected subcubic graphs $G$ with 
$\nu_s(G)=\nu_{ur}(G)$}\label{sec3}

The following lemma captures the sufficiency part of Theorem \ref{th_char}.
\begin{lemma} \label{lemma_suff}
If $G\in\mathcal{B}$,
then $\nu_s(G) = \nu_{ur}(G)$.
\end{lemma}
\begin{proof}
Let $G\in \mathcal{B}$ 
be such that $n(G)\in \{ 3k,3k+1,3k+2\}$, that is, 
$G$ arises from $L_k$ by adding at most two vertices 
and some edges.
Since 
$$\left\{u_{2i}w_{2i} : i \in \left\lbrack \left\lfloor \frac{k}{2} \right\rfloor\right \rbrack \right\} \cup
\left\{v_{2i-1}w_{2i-1} : i \in \left\lbrack \left\lceil \frac{k}{2} \right\rceil \right\rbrack \right\}$$
is an induced the matching in $G$, we have $\nu_s(G)\geq k$.
In order to complete the proof,
it suffices to show that $\nu_{ur}(G)\leq k$.
Therefore, we suppose, for a contradiction, 
that $G$ is such that $\nu_{ur}(G)>k$, 
and that the order $n(G)$ of $G$ is as small as possible.
It is easy to verify that $k\geq 3$.
Let $M$ be a maximum uniquely restricted matching in $G$.

We consider different cases.

\medskip

\noindent {\bf Case 1.} {\it $G \simeq L_k$.}

\medskip

\noindent If $M$ contains at most one edge incident with $u_k$ or $v_k$, 
then the graph $G'=G-\{ u_k,v_k,w_k\}$ has a uniquely restricted matching 
$M'=M\cap E(G')$ of size more than $k-1$.
Since $G' \simeq L_{k-1}$, we obtain a contradiction to the choice of $G$.
Hence, by symmetry, we may assume that $u_{k}u_{k-1}\in M$,
and that either $v_kv_{k-1}\in M$ or $v_kw_k\in M$.
If $v_kv_{k-1}\in M$,
then the graph $G''=G-\{ u_{k-1},v_{k-1},w_{k-1},u_k,v_k,w_k\}$ 
has a uniquely restricted matching 
$M''=M\cap E(G'')$ of size more than $k-2$.
Since $G'' \simeq L_{k-2}$, we obtain a contradiction to the choice of $G$.
Hence, $v_kw_k\in M$.
Since $M$ is uniquely restricted, we obtain $v_{k-1}w_{k-1}\not\in M$,
and $M'=(M\setminus \{ u_ku_{k-1},v_kw_k\})\cup \{ u_{k-1}w_{k-1}\}$
is a uniquely restricted matching of $G'=G-\{ u_k,v_k,w_k\}$.
Since $M'$ has size more than $k-1$, and $G' \simeq L_{k-1}$, 
we obtain a contradiction to the choice of $G$.

\medskip

\noindent {\bf Case 2.} {\it $G\in {\cal B}_1$.}

\medskip

\noindent If $M$ intersects $\{ u_1w_1',v_1w_1',u_1v_1\}$,
then adding either $u_1w_1$ or $v_1w_1$ to 
$M\setminus \{ u_1w_1',v_1w_1',u_1v_1\}$
yields a uniquely restricted matching in $G$
that 
does not intersect $\{ u_1w_1',v_1w_1',u_1v_1\}$, and
has the same size as $M$.
By symmetry, we may assume that
$M$ does not contain any edge in $E(L_k')\setminus E(L_k)$,
that is, $M$ is a uniquely restricted matching of size more than $k$
of the induced subgraph $L_k$ of $G$, and
we obtain a contradiction to the choice of $G$.

\medskip

\noindent {\bf Case 3.} {\it $G\in {\cal B}_2$ and $k$ is odd.}

\medskip

\noindent $G$ arises from $L_k$ by adding the edges $v_1u_k$ and $u_1v_k$.
In view of Case 1, we may assume, by symmetry, that $v_1u_k \in M$.
Since $G - \{v_1,w_1,u_1, v_k,w_k,u_k\} \simeq L_{k-2}$, 
it follows that 
every maximum uniquely restricted matching in $G$,
and, hence, also $M$, 
contains at least three edges incident with 
a vertex in $\{v_1,u_1, v_k,u_k\}$.
This implies $u_1v_k \not\in M$. 
If $u_1w_1 \in M$, then $v_{k-1}v_k \in M$ and $u_{k-1}w_{k-1}\not\in M$,
which implies that the matching 
$(M\cup\{v_{k-1}w_{k-1}\})\setminus \{v_{k-1}v_k \}$ 
is uniquely restricted in $G$, a contradiction.
Hence, by symmetry, we may assume that $u_1u_2, v_{k-1}v_k\in M$.

Let $M^\prime = (M \cup \{u_kw_k\}) \setminus \{v_1u_k\}$, and 
$M^{\prime\prime} = (M \cup \{v_1w_1\}) \setminus \{v_1u_k\}$. 

Since $M'$ and $M''$ are matchings in $G-\{v_1u_k,u_1v_k\} \simeq L_k$ of size more than $k$, 
Case 1 implies that 
both matchings $M^\prime$ and $M^{\prime\prime}$ 
are not uniquely restricted, that is, 
there is an $M^\prime$-alternating cycle $v_kP^\prime u_kw_kv_k$, 
and an $M^{\prime\prime}$-alternating cycle $u_1P^{\prime\prime}v_1w_1u_1$,
where $P^\prime$ and $P^{\prime\prime}$ are suitable paths.
In view of the structure of $G$,
we obtain that 
$$P^\prime=v_kv_{k-1}\ldots v_jw_ju_ju_{j+1}\ldots u_{k-1}u_k$$
and
$$P^{\prime\prime}=u_1u_2\ldots u_iw_iv_iv_{i-1}\ldots v_2v_1$$
for suitable incides $i$ and $j$ with $1<i,j<k$.
Furthermore, the structure of $G$ implies $i<j$,
which implies the contradiction that the cycle 
$u_1P^{\prime\prime}v_1u_kP^\prime v_ku_1$ is $M$-alternating.

\medskip

\noindent {\bf Case 4.} {\it $G\in {\cal B}_2$ and $k$ is even.}

\medskip

\noindent $G$ arises from $L_k$ by adding the edges $u_1u_k$ and $v_1v_k$.
Arguing similarly as in Case 3, 
we may assume that $u_1u_2,u_{k-1}u_k,v_1v_k\in M$.
Let $M^\prime = (M \cup \{v_kw_k\}) \setminus \{v_1v_k\}$, and 
$M^{\prime\prime} = (M \cup \{v_1w_1\}) \setminus \{v_1v_k\}$. 

Since $M'$ and $M''$ are matchings in $G-\{v_1v_k,u_1u_k\} \simeq L_k$ of size more than $k$, 
Case 1 implies that 
both matchings $M^\prime$ and $M^{\prime\prime}$ 
are not uniquely restricted, that is, 
there is an $M^\prime$-alternating cycle $u_kP^\prime v_kw_ku_k$ 
and an $M^{\prime\prime}$-alternating cycle $u_1P^{\prime\prime}v_1w_1u_1$,
where $P^\prime$ and $P^{\prime\prime}$ are suitable paths.
In view of the structure of $G$,
we obtain that 
$$P^\prime=u_ku_{k-1}\ldots u_jw_jv_jv_{j+1}\ldots v_{k-1}v_k$$
and
$$P^{\prime\prime}=u_1u_2\ldots u_iw_iv_iv_{i-1}\ldots v_2v_1$$
for suitable incides $i$ and $j$ with $1<i,j<k$.
Again, the structure of $G$ implies $i < j$, 
which implies the contradiction that the cycle 
$u_1P^{\prime\prime}v_1v_kP^\prime u_ku_1$ is $M$-alternating.
\end{proof}

\begin{proof}[Proof of Theorem \ref{th_char}]
Let $G$ be a $2$-connected subcubic graph of order at least $21$.
If $G$ is isomorphic to a graph in $\mathcal{B}$,
then Lemma \ref{lemma_suff} implies $\nu_s(G) = \nu_{ur}(G)$.
In order to complete the proof,
we assume $\nu_s(G) = \nu_{ur}(G)$, 
and deduce that 
$G$ is isomorphic to a graph in $\mathcal{B}$.
Let $M$ be a maximum induced matching in $G$.
By Lemma \ref{lemma0}(i), $G-V(M)$ has maximum degree at most $2$.
A pair of disjoint edges $u_1u_2$ and $v_1v_2$ with $u_1v_1 \in M$ is called a {\it local pair}.

We consider several cases and subcases.
Within each (sub)case, we will --- sometimes tacitly --- assume 
that the local configurations considered in the previous (sub)cases
are no longer possible.
In each (sub)case, we conclude that 
\begin{itemize}
\item either Lemma \ref{lemma0} fails, which is a contradiction,
\item or $n(G)\leq 20$, which is a contradiction,
\item or $G\in {\cal B}$ as desired.
\end{itemize}

\noindent {\bf Case 1.} {\it Some component $H$ of $G-V(M)$ has order at least $5$.}

\medskip

\noindent Since $H$ has maximum degree at most $2$, 
$H$ contains a path $u_1u_2u_3u_4u_5$. 
By Lemma \ref{lemma0}(i) for the edge $u_1u_2$, 
there is a $4$-cycle $u_1u_2v_2v_1u_1$ with $v_1v_2 \in M$.
By Lemma \ref{lemma0}(i) for the edge $u_2u_3$,
$u_3$ is adjacent to $v_1$.
By Lemma \ref{lemma0}(i) for the edge $u_3u_4$,
$u_4$ is adjacent to $v_2$.
Now, since $G$ is subcubic,
Lemma \ref{lemma0}(i) fails for the edge $u_4u_5$.
See Figure \ref{figcase1} for an illustration.

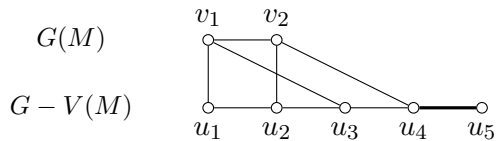
\begin{figure}[H]
\centering\tiny
\begin{tikzpicture}[scale = 0.9] 
	    \node[label=below:\normalsize $u_1$] (u1) at (0,0) {};
	    \node[label=below:\normalsize $u_2$] (u2) at (1,0) {};
	    \node[label=below:\normalsize $u_3$] (u3) at (2,0) {};
	    \node[label=below:\normalsize $u_4$] (u4) at (3,0) {};
	    \node[label=below:\normalsize $u_5$] (u5) at (4,0) {};
	    \node[label=above:\normalsize $v_1$] (v1) at (0,1) {};
	    \node[label=above:\normalsize $v_2$] (v2) at (1,1) {};
	    
	    \foreach \from/\to in {u1/u2, u2/u3, u3/u4, u1/v1, u3/v1, u2/v2, u4/v2, v1/v2}
	    \draw [-] (\from) -- (\to);
	    
	    \draw[-,very thick] (u4) -- (u5);
	    
	    \pgftext[x=-2cm,y=0cm] {\normalsize $G-V(M)$};
	    \pgftext[x=-2cm,y=1cm] {\normalsize $G(M)$};
	    
\end{tikzpicture}
\caption{The final situation in Case 1.} \label{figcase1}
\end{figure}

\noindent {\bf Case 2.} {\it Some component $H$ of $G-V(M)$ has order $4$.}

\medskip

\noindent Similarly as in Case 1, 
$H$ contains a path $u_1u_2u_3u_4$,
and $M$ contains an edge $v_1v_2$ such that 
$v_1$ is adjacent to $u_1$ and $u_3$, and
$v_2$ is adjacent to $u_2$ and $u_4$.
By Lemma \ref{lemma0}(ii) 
for the local pair $v_1u_1$ and $v_2u_4$,
there is an edge $v_3v_4\in M$ distinct from $v_1v_2$
such that 
$u_1$ is adjacent to $v_3$, 
and
$u_4$ is adjacent to $v_4$.
Since $G$ is $2$-connected, and $n(G)>8$, 
$v_4$ has a neighbor $x$ distinct from $v_3$ and $u_4$.
Now, Lemma \ref{lemma0}(ii) fails for the local pair $v_3u_1$ and $v_4x$.
See Figure \ref{figcase2} for an illustration.

\begin{figure}[H]
\centering\tiny
\begin{tikzpicture}[scale = 0.9] 
	    \node[label=below:\normalsize $u_1$] (u1) at (0,0) {};
	    \node[label=below:\normalsize $u_2$] (u2) at (1,0) {};
	    \node[label=below:\normalsize $u_3$] (u3) at (2,0) {};
	    \node[label=below:\normalsize $u_4$] (u4) at (3,0) {};
	    \node[label=below:\normalsize $x$] (x) at (5,0) {};
	    \node[label=above:\normalsize $v_1$] (v1) at (0,1) {};
	    \node[label=above:\normalsize $v_2$] (v2) at (1,1) {};
	    \node[label=above:\normalsize $v_3$] (v3) at (3,1) {};
	    \node[label=above:\normalsize $v_4$] (v4) at (4,1) {};
	    
	    \foreach \from/\to in {u1/u2, u2/u3, u3/u4, u1/v1, u3/v1, u2/v2, u4/v2, v1/v2, v3/v4, v4/u4}
	    \draw [-] (\from) -- (\to);
	    
	    \draw[-,very thick] (v3) -- (u1);
	    \draw[-,very thick] (v4) -- (x);
	    
	    \pgftext[x=-2cm,y=0cm] {\normalsize $G-V(M)$};
	    \pgftext[x=-2cm,y=1cm] {\normalsize $G(M)$};
	    
\end{tikzpicture}
\caption{The final situation in Case 2.} \label{figcase2}
\end{figure}
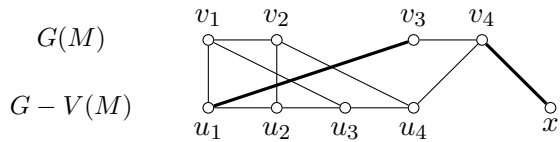

\noindent {\bf Case 3.} {\it Some component $H$ of $G-V(M)$ has order $3$.}

\medskip

\noindent Similarly as in Cases 1 and 2, 
$H$ contains a path $u_1u_2u_3$,
and $M$ contains an edge $v_1v_2$ such that 
$v_1$ is adjacent to $u_1$ and $u_3$, and
$v_2$ is adjacent to $u_2$.
Since $G$ is $2$-connected, and $n(G)>5$, 
$v_2$ is not adjacent to $u_1$ or $u_3$,
and $u_3$ has a neighbor $v_3$ distinct from $u_2$ and $v_1$.
By Case 2, $v_3\in V(M)$, and $M$ contains an edge $v_3v_4$.
Since $G$ is $2$-connected, 
$v_4$ has a neighbor $x$ distinct from $v_3$.
If $v_4$ is adjacent to $u_1$, 
then Lemma \ref{lemma0}(ii) fails for the local pair $v_3u_3$ and $v_4u_1$,
that is, $x\not= u_1$.
By Lemma \ref{lemma0}(ii) for the local pair $v_3u_3$ and $v_4x$,
$x$ is adjacent to $v_2$.
If $v_4$ has a neighbor $y$ distinct from $v_3$ and $x$,
then Lemma \ref{lemma0}(ii) fails for the local pair $v_3u_3$ and $v_4y$.
Hence, $v_4$ has degree $2$ in $G$.
Since $G$ is $2$-connected, and $n(G)>8$,
$v_3$ is not adjacent to $u_1$.
By Case 2, and Lemma \ref{lemma0}(ii) for the local pair $v_1u_1$ and $v_2x$,
there is a $6$-cycle $v_1u_1v_5v_6xv_2v_1$, 
where $v_5v_6 \in M$ is distinct from $v_3v_4$.
By symmetry between $v_4$ and $v_6$,
$v_6$ has degree $2$ in $G$.
Since $G$ is $2$-connected, and $n(G)>10$, 
we may assume, by symmetry between $v_3$ and $v_5$,
that $v_3$ has a neighbor $y$ distinct from $v_4$ and $u_3$.
By Lemma \ref{lemma0}(ii) for the local pair $v_3y$ and $v_4x$,
$v_5$ is adjacent to $y$.
Since $G$ is $2$-connected, $n(G)=11$.
See Figure \ref{figcase3} for an illustration.

 \begin{figure}[H]
\centering\tiny
\begin{tikzpicture}[scale = 0.9] 
	    \node[label=below:\normalsize $u_1$] (u1) at (0,0) {};
	    \node[label=below:\normalsize $u_2$] (u2) at (1,0) {};
	    \node[label=below:\normalsize $u_3$] (u3) at (2,0) {};
	    \node[label=below:\normalsize $x$] (x) at (4,0) {};
	    \node[label=below:\normalsize $y$] (y) at (6,0) {};
	    \node[label=above:\normalsize $v_1$] (v1) at (0,1) {};
	    \node[label=above:\normalsize $v_2$] (v2) at (1,1) {};
	    \node[label=above:\normalsize $v_3$] (v3) at (3,1) {};
	    \node[label=above:\normalsize $v_4$] (v4) at (4,1) {};
	    \node[label=above:\normalsize $v_5$] (v5) at (6,1) {};
	    \node[label=above:\normalsize $v_6$] (v6) at (7,1) {};
	    
	    \foreach \from/\to in {u1/u2, u2/u3, u1/v1, u3/v1, u2/v2, v2/x, x/v4, x/v6, y/v3, y/v5, v1/v2, v3/u3, v3/v4, v5/v6, v5/u1}
	    \draw [-] (\from) -- (\to);
	    
	    \pgftext[x=-2cm,y=0cm] {\normalsize $G-V(M)$};
	    \pgftext[x=-2cm,y=1cm] {\normalsize $G(M)$};
	    
\end{tikzpicture}
\caption{The final situation in Case 3.} \label{figcase3}
\end{figure}
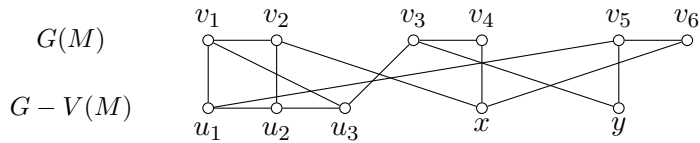

\noindent {\bf Case 4.} {\it $G-V(M)$ contains an edge $u_1u_2$.}

\medskip

\noindent By Lemma \ref{lemma0}(i) for the edge $u_1u_2$, 
there is a $4$-cycle $u_1u_2v_2v_1u_1$ with $v_1v_2 \in M$.
If $v_1$ is adjacent to $u_2$, then, 
since $G$ is $2$-connected, and $n(G)>4$,
$M$ contains an edge $v_3v_4$ distinct from $v_1v_2$
such that $u_1$ is adjacent to $v_3$.
Since $G$ is $2$-connected,
$v_4$ has a neighbor $x$ distinct from $v_3$.
By Lemma \ref{lemma0}(ii) 
for the local pair $v_3u_1$ and $v_4x$,
$x$ is adjacent to $v_2$.
Now, Lemma \ref{lemma0}(ii) fails
for the local pair $v_1u_2$ and $v_2x$.
Hence, by symmetry, $v_1$ is not adjacent to $u_2$,
and $v_2$ is not adjacent to $u_1$.

\medskip

\noindent {\bf Case 4.1.} 
{\it $v_2$ has a neighbor $u_3$ that belongs to an edge $u_3u_4$ of $G-V(M)$ distinct from $u_1u_2$, and $u_4$ is adjacent to $v_1$.}

\medskip

\noindent By Lemma \ref{lemma0}(ii) 
for the local pair $v_1u_1$ and $v_2u_3$,
there is a $6$-cycle $v_1u_1v_3v_4u_3v_2v_1$
with $v_3v_4 \in M$. 
By Lemma \ref{lemma0}(ii) 
for the local pair $v_1u_4$ and $v_2u_2$,
there is a $6$-cycle 
$v_1u_4v_5v_6u_2v_2v_1$ 
with $v_5v_6 \in M$.
Since $n(G)>8$, 
the edges $v_3v_4$ and $v_5v_6$ are distinct.
Lemma \ref{lemma0}(ii) 
implies that the vertices 
$v_3$, 
$v_4$, 
$v_5$, and 
$v_6$ have degree $2$ in $G$,
which implies $n(G) = 10$.
See Figure \ref{figcase4.1} for an illustration.

 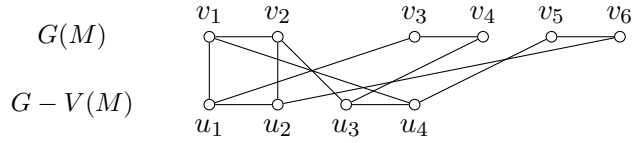
\begin{figure}[H]
\centering\tiny
\begin{tikzpicture}[scale = 0.9] 
	    \node[label=below:\normalsize $u_1$] (u1) at (0,0) {};
	    \node[label=below:\normalsize $u_2$] (u2) at (1,0) {};
	    \node[label=below:\normalsize $u_3$] (u3) at (2,0) {};
	    \node[label=below:\normalsize $u_4$] (u4) at (3,0) {};
	    \node[label=above:\normalsize $v_1$] (v1) at (0,1) {};
	    \node[label=above:\normalsize $v_2$] (v2) at (1,1) {};
	    \node[label=above:\normalsize $v_3$] (v3) at (3,1) {};
	    \node[label=above:\normalsize $v_4$] (v4) at (4,1) {};
	    \node[label=above:\normalsize $v_5$] (v5) at (5,1) {};
	    \node[label=above:\normalsize $v_6$] (v6) at (6,1) {};
	    
	    \foreach \from/\to in {v1/v2, v3/v4, v5/v6, u1/u2, u3/u4, u1/v1, u1/v3, u2/v2, u2/v6, u3/v2, u3/v4, u4/v1, u4/v5}
	    \draw [-] (\from) -- (\to);
	    
	    \pgftext[x=-2cm,y=0cm] {\normalsize $G-V(M)$};
	    \pgftext[x=-2cm,y=1cm] {\normalsize $G(M)$};
	    
\end{tikzpicture}
\caption{The final situation in Case 4.1.} \label{figcase4.1}
\end{figure}

\noindent {\bf Case 4.2.} 
{\it $v_2$ has a neighbor $u_3$ that belongs to an edge $u_3u_4$ of $G-V(M)$ distinct from $u_1u_2$, 
and $u_4$ is not adjacent to $v_1$.}

\medskip

\noindent By Lemma \ref{lemma0}(i) for the edge $u_3u_4$, 
there is a $4$-cycle $u_3u_4v_4v_3u_3$, 
where $v_3v_4 \in M$ is distinct from $v_1v_2$. 
By Lemma \ref{lemma0}(ii) 
for the local pair $v_1u_1$ and $v_2u_3$,
$u_1$ is adjacent to $v_4$.
If $u_4$ is adjacent to $v_3$, 
then Lemma \ref{lemma0}(ii)
fails for the local pair $v_3u_4$ and $v_4u_1$.
Hence, $u_4$ is not adjacent to $v_3$.
If $u_2$ is adjacent to $v_3$, 
then Lemma \ref{lemma0}(ii)
fails for the local pair $v_3u_2$ and $v_4u_4$.
Hence, $u_2$ is not adjacent to $v_3$.

If $u_2$ has a neighbor $v_5$ distinct from $u_1$ and $v_2$,
then there is an edge $v_5v_6\in M$ 
distinct from $v_1v_2$ and $v_3v_4$.
If $v_6$ is adjacent to $u_4$, 
then Lemma \ref{lemma0}(ii) fails
for the local pair $v_5u_2$ and $v_6u_4$.
Hence, $v_6$ is not adjacent to $u_4$.
Since $G$ is $2$-connected, 
$v_6$ has a neighbor $x$ distinct from $v_5$.
By Lemma \ref{lemma0}(ii) 
for the local pair $v_5u_2$ and $v_6x$,
$x$ is adjacent to $v_1$.
Now, Lemma \ref{lemma0}(ii) fails
for the local pair $v_1x$ and $v_2u_3$.
See the left part of Figure \ref{figcase4.2}.
Hence, by symmetry between $u_2$ and $u_4$,
$u_2$ and $u_4$ have degree $2$ in $G$.

Since $G$ is $2$-connected, and $n(G)>8$,
$v_1$ has a neighbor $x$ distinct from $u_1$ and $v_2$.
Now, Lemma \ref{lemma0}(ii) fails
for the local pair $v_1x$ and $v_2u_3$.
See the right part of Figure \ref{figcase4.2}.

\begin{figure}[H]
\begin{minipage}{0.49\textwidth}
\centering\tiny
\begin{tikzpicture}[scale = 0.9] 
	    \node[label=below:\normalsize $u_1$] (u1) at (0,0) {};
	    \node[label=below:\normalsize $u_2$] (u2) at (1,0) {};
	    \node[label=below:\normalsize $u_3$] (u3) at (2,0) {};
	    \node[label=below:\normalsize $u_4$] (u4) at (3,0) {};
	    \node[label=above:\normalsize $v_1$] (v1) at (0,1) {};
	    \node[label=above:\normalsize $v_2$] (v2) at (1,1) {};
	    \node[label=above:\normalsize $v_3$] (v3) at (2,1) {};
	    \node[label=above:\normalsize $v_4$] (v4) at (3,1) {};
	    \node[label=above:\normalsize $v_5$] (v5) at (-1,1) {};
	    \node[label=above:\normalsize $v_6$] (v6) at (-2,1) {};
	    \node[label=below:\normalsize $x$] (x) at (-2,0) {};
	    
	    \foreach \from/\to in {v1/v2, v3/v4, v5/v6, u1/u2, u3/u4, u1/v1, u1/v4, u2/v2, u2/v5, u3/v3, u4/v4, x/v6}
	    \draw [-] (\from) -- (\to);
	    
	    \draw[-,very thick] (v1) -- (x);
	    \draw[-,very thick] (v2) -- (u3);
	    
	    \pgftext[x=-4cm,y=0cm] {\normalsize $G-V(M)$};
	    \pgftext[x=-4cm,y=1cm] {\normalsize $G(M)$};
	    
\end{tikzpicture}
\end{minipage}
\begin{minipage}{0.49\textwidth}
\centering\tiny
\begin{tikzpicture}[scale = 0.9] 
	    \node[label=below:\normalsize $u_1$] (u1) at (0,0) {};
	    \node[label=below:\normalsize $u_2$] (u2) at (1,0) {};
	    \node[label=below:\normalsize $u_3$] (u3) at (2,0) {};
	    \node[label=below:\normalsize $u_4$] (u4) at (3,0) {};
	    \node[label=above:\normalsize $v_1$] (v1) at (0,1) {};
	    \node[label=above:\normalsize $v_2$] (v2) at (1,1) {};
	    \node[label=above:\normalsize $v_3$] (v3) at (2,1) {};
	    \node[label=above:\normalsize $v_4$] (v4) at (3,1) {};
	    \node[label=below:\normalsize $x$] (x) at (-1,0) {};

	    \foreach \from/\to in {v1/v2, v3/v4, u1/u2, u3/u4, u1/v1, u2/v2, u3/v3, u4/v4,u1/v4}
	    \draw [-] (\from) -- (\to);
	    
	    \draw[-,very thick] (v1) -- (x);
	    \draw[-,very thick] (v2) -- (u3);
	    
\end{tikzpicture}
\end{minipage}
\caption{Two situations in Case 4.2.} \label{figcase4.2}
\end{figure}
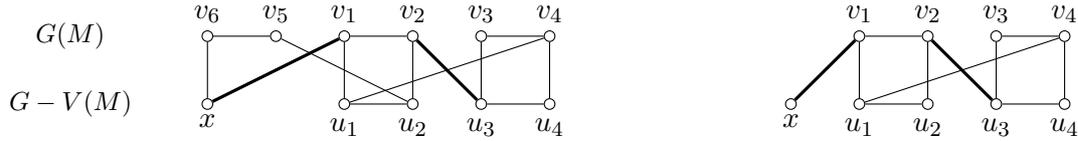

\noindent {\bf Case 4.3.} 
{\it $v_2$ has a neighbor $u_3$ 
that is an isolated vertex of $G-V(M)$.}

\medskip

\noindent By Lemma \ref{lemma0}(ii) 
for the local pair $v_1u_1$ and $v_2u_3$,
there is a $6$-cycle $v_1u_1v_3v_4u_3v_2v_1$
with $v_3v_4 \in M$. 

\medskip

\noindent {\bf Case 4.3.1.} 
{\it $v_4$ has a neighbor $x$ distinct from $v_3$ and $u_3$.}

\medskip

\noindent By Lemma \ref{lemma0}(ii) 
for the local pair $v_3u_1$ and $v_4x$,
$x=u_2$.
Since $G$ is $2$-connected, and $n(G)>7$,
$v_1$ is not adjacent to $u_3$.

If $v_1$ has a neighbor $y$ distinct from $u_1$ and $v_2$,
then, by Lemma \ref{lemma0}(ii) 
for the local pair $v_1y$ and $v_2u_2$,
$y$ is adjacent to $v_3$.
Since the matching 
$(M \cup \{ v_1y, v_4u_3, u_1u_2\}) 
\setminus \{v_1v_2,v_3v_4\}$
is not uniquely restricted,
there is an edge $v_5v_6\in M$ distinct from $v_1v_2$ and $v_3v_4$ such that $u_3$ is adjacent to $v_5$.
Since $G$ is $2$-connected,
$v_6$ has a neighbor $z$ distinct from $v_5$.
By Lemma \ref{lemma0}(ii) 
for the local pair $v_5u_3$ and $v_6z$,
$y=z$.
By Lemma \ref{lemma0}(ii),
the vertices $v_5$ and $v_6$ have degree $2$ in $G$,
which implies $n(G)=10$.
See the left part of Figure \ref{figcase4.3.1}.
Hence, $v_1$ has degree $2$ in $G$.

Since $n(G)>7$, and $G$ is $2$-connected,
there is an edge $v_5v_6\in M$ distinct from $v_1v_2$ and $v_3v_4$ such that $u_3$ is adjacent to $v_5$.
Since $G$ is $2$-connected,
$v_6$ has a neighbor $u_4$ distinct from $v_5$.
By Lemma \ref{lemma0}(ii) 
for the local pair $v_5u_3$ and $v_6u_4$,
$u_4$ is adjacent to $v_3$. 
Now, Lemma \ref{lemma0}(ii) fails
for the local pair $v_3u_4$ and $v_4u_2$.
See the right part of Figure \ref{figcase4.3.1}.

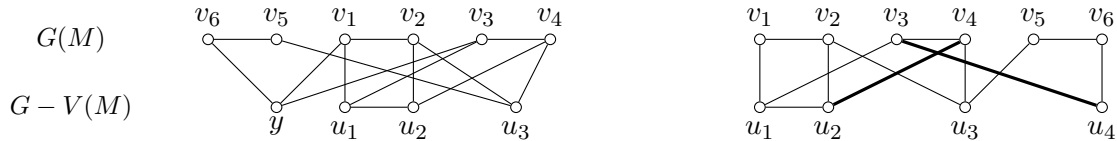
\begin{figure}[H]
\begin{minipage}{0.49\textwidth}
\centering\tiny
\begin{tikzpicture}[scale = 0.9] 
	    \node[label=below:\normalsize $u_1$] (u1) at (0,0) {};
	    \node[label=below:\normalsize $u_2$] (u2) at (1,0) {};
	    \node[label=below:\normalsize $u_3$] (u3) at (2.5,0) {};
	    \node[label=below:\normalsize $y$] (y) at (-1,0) {};
	    \node[label=above:\normalsize $v_1$] (v1) at (0,1) {};
	    \node[label=above:\normalsize $v_2$] (v2) at (1,1) {};
	    \node[label=above:\normalsize $v_3$] (v3) at (2,1) {};
	    \node[label=above:\normalsize $v_4$] (v4) at (3,1) {};
	    \node[label=above:\normalsize $v_5$] (v5) at (-1,1) {};
	    \node[label=above:\normalsize $v_6$] (v6) at (-2,1) {};
	    
	    \foreach \from/\to in {v1/v2, v3/v4, v5/v6, u1/u2, u1/v1, u1/v3, u2/v2, u2/v4, u3/v2, u3/v4, u3/v5, y/v1, y/v3, y/v6}
	    \draw [-] (\from) -- (\to);
	    
	    \pgftext[x=-4cm,y=0cm] {\normalsize $G-V(M)$};
	    \pgftext[x=-4cm,y=1cm] {\normalsize $G(M)$};
	    
\end{tikzpicture}
\end{minipage}
\begin{minipage}{0.49\textwidth}
\centering\tiny
\begin{tikzpicture}[scale = 0.9] 
	    \node[label=below:\normalsize $u_1$] (u1) at (0,0) {};
	    \node[label=below:\normalsize $u_2$] (u2) at (1,0) {};
	    \node[label=below:\normalsize $u_3$] (u3) at (3,0) {};
	    \node[label=below:\normalsize $u_4$] (u4) at (5,0) {};
	    \node[label=above:\normalsize $v_1$] (v1) at (0,1) {};
	    \node[label=above:\normalsize $v_2$] (v2) at (1,1) {};
	    \node[label=above:\normalsize $v_3$] (v3) at (2,1) {};
	    \node[label=above:\normalsize $v_4$] (v4) at (3,1) {};
	    \node[label=above:\normalsize $v_5$] (v5) at (4,1) {};
	    \node[label=above:\normalsize $v_6$] (v6) at (5,1) {};

	    \foreach \from/\to in {v1/v2, v3/v4, v5/v6, u1/u2, u1/v1, u1/v3, u2/v2, u3/v2, u3/v4, u3/v5, u4/v6}
	    \draw [-] (\from) -- (\to);
	    
	    \draw[-,very thick] (v3) -- (u4);
	    \draw[-,very thick] (v4) -- (u2);
	    
\end{tikzpicture}
\end{minipage}
\caption{Two situations in Case 4.3.1.} \label{figcase4.3.1}
\label{fig9}
\end{figure}

\medskip

\noindent {\bf Case 4.3.2.} 
{\it $v_4$ has degree $2$ in $G$,
and $v_1$ has a neighbor $x$ distinct from $u_1$ and $v_2$.}

\medskip

\noindent Since $v_1$ is not adjacent to $u_2$, $x\not=u_2$.
If $x = u_3$, then,
by Lemma \ref{lemma0}(ii) 
for the local pair $v_1u_3$ and $v_2u_2$,
$u_2$ is adjacent to $v_3$.
Since $G$ is $2$-connected, 
this implies $n(G)=7$.
Hence, $x$ is distinct from $u_2$ and $u_3$.

Since $\{ x,u_2,v_4\}$ is an independent set,
by Lemma \ref{lemma0}(ii) 
for the local pair $v_1x$ and $v_2u_2$,
there is a $6$-cycle $v_1xv_6v_5u_2v_2v_1$, 
where $v_5v_6 \in M$ is distinct from $v_1v_2$ and $v_3v_4$.
If $u_3$ is adjacent to $v_6$,
then Lemma \ref{lemma0}(ii) fails
for the local pair $v_6u_3$ and $v_5u_2$.
Hence, $u_3$ is not adjacent to $v_6$.
By Lemma \ref{lemma0}(ii) 
for the local pair $v_1x$ and $v_2u_3$,
\begin{itemize}
\item either $x$ is adjacent to $v_3$,
\item or $u_3$ is adjacent to $v_5$,
\item or there is a $6$-cycle $v_1xv_7v_8u_3v_2v_1$, 
where $v_7v_8 \in M$
is not in $\{v_1v_2,v_3v_4, v_5v_6 \}$. 
\end{itemize}
First, we assume that $x$ is adjacent to $v_3$.
If $v_6$ has a neighbor $y$ distinct from $v_5$ and $x$,
then Lemma \ref{lemma0}(ii) fails 
for the local pair $v_6y$ and $v_5u_2$.
Hence, $v_6$ has degree $2$ in $G$.
Since $v_4$ has degree $2$ in $G$,
$v_5$ has a neighbor $y$ distinct from $u_2$ and $v_6$.
Since $n(G)>10$, $y\not=u_3$,
and Lemma \ref{lemma0}(ii) fails
for the local pair $v_5y$ and $v_6x$.
See the left part of Figure \ref{figcase4.3.2}.

Next, we assume that $x$ is not adjacent to $v_3$,
but that $u_3$ is adjacent to $v_5$.
If $v_3$ has a neighbor $y$ distinct from $u_1$ and $v_4$,
then $y\not=x$, and, by Lemma \ref{lemma0}(ii) 
for the local pair $v_3y$ and $v_4u_3$,
$y$ is adjacent to $v_6$.
Now, Lemma \ref{lemma0}(ii) fails 
for the local pair $v_6y$ and $v_5u_2$.
Hence, $v_3$ has degree $2$ in $G$.
Since $v_4$ has degree $2$ in $G$,
$v_6$ has a neighbor $z$ distinct from $x$ and $v_5$.
Now, Lemma \ref{lemma0}(ii) fails 
for the local pair $v_6z$ and $v_5u_2$.

Finally, we assume that 
$x$ is not adjacent to $v_3$,
$u_3$ is not adjacent to $v_5$, but 
there is a $6$-cycle $v_1xv_7v_8u_3v_2v_1$, 
where $v_7v_8 \in M$
is distinct from $\{v_1v_2,v_3v_4, v_5v_6 \}$. 
Since $x$ is isolated in $G-V(M)$,
the vertices $x$ and $u_3$ are symmetric.
In view of the previous cases, 
and the symmetry between $v_4$ and $v_6$, 
we obtain that $v_6$ has degree $2$ in $G$.
If $v_5$ has a neighbor $y$ distinct from $u_2$ and $v_6$,
then, by Lemma \ref{lemma0}(ii) for 
the local pair $v_5y$ and $v_6x$,
$v_8$ is adjacent to $y$.
If $v_8$ has a neighbor $y'$ distinct from $u_3$ and $v_7$,
then, by Lemma \ref{lemma0}(ii) for 
the local pair $v_8y'$ and $v_7x$,
$v_5$ is adjacent to $y'$.
Hence, $v_5$ and $v_8$ 
either both have degree $2$
or degree $3$ and a common neighbor.
Similarly, $v_3$ and $v_7$ 
either both have degree $2$
or degree $3$ a common neighbor.
If $v_5$ and $v_8$ have a common neighbor $u$,
and $v_3$ and $v_7$ have a common neighbor $v$,
then, since $n(G) > 14$, $u$ and $v$ are not adjacent.
By Lemma \ref{lemma0}(ii) for 
the local pair $v_8u$ and $v_7v$,
there is a $6$-cycle $v_7vv_9v_{10}uv_8v_7$ 
with $v_9v_{10} \in M$. 
Since, by Lemma \ref{lemma0}(ii),
$v_9$ and $v_{10}$ have degree $2$ in $G$,
we obtain $n(G)=16$.
See the right part of Figure \ref{figcase4.3.2}.
Hence, by symmetry, 
we may assume that $v_5$ and $v_8$ 
have a common neighbor $u$,
and that $v_3$ and $v_7$ have degree $2$ in $G$.
Since $G$ is $2$-connected, 
we obtain $n(G)=13$.

\begin{figure}[H]
\begin{minipage}{0.49\textwidth}
\centering\tiny
\begin{tikzpicture}[scale = 0.9] 
	    \node[label=below:\normalsize $u_1$] (u1) at (0,0) {};
	    \node[label=below:\normalsize $u_2$] (u2) at (1,0) {};
	    \node[label=below:\normalsize $u_3$] (u3) at (2.5,0) {};
	    \node[label=below:\normalsize $x$] (x) at (-1.5,0) {};
	    \node[label=below:\normalsize $y$] (y) at (-2.5,0) {};
	    \node[label=above:\normalsize $v_1$] (v1) at (0,1) {};
	    \node[label=above:\normalsize $v_2$] (v2) at (1,1) {};
	    \node[label=above:\normalsize $v_3$] (v3) at (2,1) {};
	    \node[label=above:\normalsize $v_4$] (v4) at (3,1) {};
	    \node[label=above:\normalsize $v_5$] (v5) at (-1,1) {};
	    \node[label=above:\normalsize $v_6$] (v6) at (-2,1) {};
	    
	    \foreach \from/\to in {v1/v2, v3/v4, v5/v6, u1/u2, u1/v1, u1/v3, u2/v2, u2/v5, u3/v2, u3/v4, x/v1, x/v3}
	    \draw [-] (\from) -- (\to);
	    
	    \draw[-,very thick] (v6) -- (x);
	    \draw[-,very thick] (v5) -- (y);
	    
	    \pgftext[x=-4cm,y=0cm] {\normalsize $G-V(M)$};
	    \pgftext[x=-4cm,y=1cm] {\normalsize $G(M)$};
	    
\end{tikzpicture}
\end{minipage}
\begin{minipage}{0.49\textwidth}
\centering\tiny
\begin{tikzpicture}[scale = 0.9] 
	    \node[label=below:\normalsize $u_1$] (u1) at (0,0) {};
	    \node[label=below:\normalsize $u_2$] (u2) at (1,0) {};
	    \node[label=below:\normalsize $u_3$] (u3) at (2.5,0) {};
	    \node[label=below:\normalsize $x$] (x) at (-1.5,0) {};
	    \node[label=below:\normalsize $u$] (u) at (3.5,0) {};
	    \node[label=below:\normalsize $v$] (v) at (5,0) {};
	    \node[label=above:\normalsize $v_1$] (v1) at (0,1) {};
	    \node[label=above:\normalsize $v_2$] (v2) at (1,1) {};
	    \node[label=above:\normalsize $v_3$] (v3) at (2,1) {};
	    \node[label=above:\normalsize $v_4$] (v4) at (3,1) {};
	    \node[label=above:\normalsize $v_5$] (v5) at (-1,1) {};
	    \node[label=above:\normalsize $v_6$] (v6) at (-2,1) {};
	    \node[label=above:\normalsize $v_7$] (v7) at (4,1) {};
	    \node[label=above:\normalsize $v_8$] (v8) at (5,1) {};
	    \node[label=above:\normalsize $v_9$] (v9) at (6,1) {};
	    \node[label=above:\normalsize $v_{10}$] (v10) at (7,1) {};

	    \foreach \from/\to in {v1/v2, v3/v4, v5/v6, v7/v8, u1/u2, u1/v1, u1/v3, u2/v2, u2/v5, u3/v2, u3/v4, u3/v8, x/v6, x/v1, x/v7}
	    \draw [-] (\from) -- (\to);
	    
	    \foreach \from/\to in {u/v5, u/v8, u/v10, v/v3, v/v7, v/v9, v9/v10}
	    \draw [-, dashed] (\from) -- (\to);
	    
\end{tikzpicture}
\end{minipage}
\caption{Two situations in Case 4.3.2.} \label{figcase4.3.2}
\end{figure}
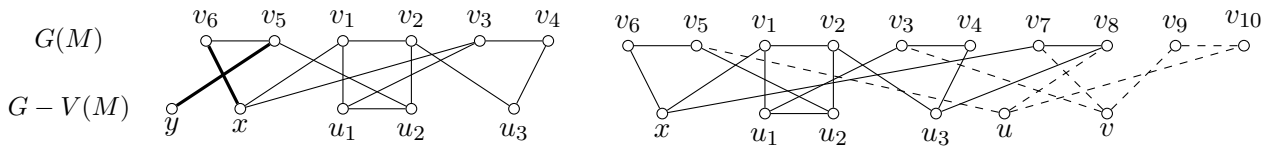

\noindent {\bf Case 4.3.3.} 
{\it $v_1$ and $v_4$ have degree $2$ in $G$.}

\medskip

\noindent If $u_2$ is adjacent to $v_3$,
then Lemma \ref{lemma0}(ii) fails
for the local pair $v_3u_2$ and $v_4u_3$.
Hence, $u_2$ is not adjacent to $v_3$.
If there is an edge $v_5v_6\in M$ distinct from $v_1v_2$ and $v_3v_4$, and $u_2$ is adjacent to $v_5$,
then $v_6$ has a neighbor $y$ distinct from $v_5$,
and Lemma \ref{lemma0}(ii) fails
for the local pair $v_6y$ and $v_5u_2$.
Hence, $u_2$ has degree $2$ in $G$.
Since $G$ is $2$-connected, and $n(G)>7$,
$v_3$ has a neighbor $u_4$ distinct from $u_1$ and $v_4$.
By Lemma \ref{lemma0}(ii) 
for the local pair $v_3u_4$ and $v_4u_3$,
there is a $6$-cycle $v_3u_4v_6v_5u_3v_4v_3$,
where $v_5v_6 \in M$ is distinct from $v_1v_2$ and $v_3v_4$.
By Lemma \ref{lemma0}(ii),
$v_6$ has degree $2$ in $G$.
Hence, for the vertex $v_5$,
we are in a similar situation as for $v_3$,
and we set up an inductive argument.

Let $k\geq 6$ be the largest even integer such that,
for every $i\in  \left\lbrack \frac{k-2}{2}  \right\rbrack \setminus \lbrack 2 \rbrack$,
\begin{itemize}
\item the vertex $v_{2i-1}$ has the neighbors $u_{i}$, $u_{i+2}$, and $v_{2i}$ 
such that $u_{i+1}$ and $u_{i+2}$ are not adjacent,
\item and $v_{2i}$ has degree $2$ in $G$, and is adjacent to $v_{2i-1}$ and $u_{i+1}$.
\end{itemize}
Note that these two conditions are satisfied for $k=6$ by the previous discussion.

By Lemma \ref{lemma0}(ii) for the local pair 
$v_{k-3}u_{\frac{k}{2}+1}$ and $v_{k-2}u_{\frac{k}{2}}$,
there is a $6$-cycle
$v_{k-2} u_{\frac{k}{2}} v_{k-1} v_{k} u_{\frac{k}{2}+1} v_{k-3} v_{k-2}$
with $v_{k-1} v_{k}\in M$. 
By Lemma \ref{lemma0}(ii),
$v_k$ has degree $2$ in $G$,
that is,
$v_{k-1}$ and $u_{\frac{k}{2}+1}$ are the only vertices in $G$ with further neighbors.
By the choice of $k$, Lemma \ref{lemma0}(ii), 
and since $G$ is $2$-connected,
\begin{itemize}
\item either $v_{k-1}$ and $u_{\frac{k}{2}+1}$ both have degree $2$ in $G$,
\item or $v_{k-1}$ and $u_{\frac{k}{2}+1}$ are adjacent,
\item or $v_{k-1}$ and $u_{\frac{k}{2}+1}$
have a neighbor $u_{\frac{k}{2}+2}$.
\end{itemize}
See Figure \ref{figcase4.3.3} illustrating these options.
In the first two cases, 
it follows immediately 
that $G$ is isomorphic to a graph in ${\cal B}_1$.
In the final case, 
$u_{\frac{k}{2}+2}$ has degree $2$ in $G$, 
and also in this case 
it follows 
that $G$ is isomorphic to a graph in ${\cal B}_1$.

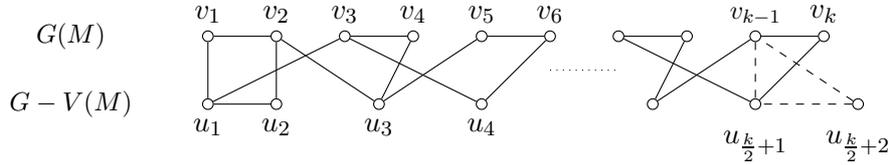
\begin{figure}[H]
\centering\tiny
\begin{tikzpicture}[scale = 0.9] 
	    \node[label=below:\normalsize $u_1$] (u1) at (0,0) {};
	    \node[label=below:\normalsize $u_2$] (u2) at (1,0) {};
	    \node[label=below:\normalsize $u_3$] (u3) at (2.5,0) {};
	    \node[label=below:\normalsize $u_4$] (u4) at (4,0) {};
	    \node[label=above:\normalsize $v_1$] (v1) at (0,1) {};
	    \node[label=above:\normalsize $v_2$] (v2) at (1,1) {};
	    \node[label=above:\normalsize $v_3$] (v3) at (2,1) {};
	    \node[label=above:\normalsize $v_4$] (v4) at (3,1) {};
	    \node[label=above:\normalsize $v_5$] (v5) at (4,1) {};
	    \node[label=above:\normalsize $v_6$] (v6) at (5,1) {};
	    
	    \foreach \from/\to in {v1/v2, v3/v4, v5/v6, u1/u2, u1/v1, u1/v3, u2/v2, u3/v2, u3/v4, u3/v5, u4/v3, u4/v6}
	    \draw [-] (\from) -- (\to);
	    
	    \draw[-, dotted] (5,0.5) -- (6,0.5);
	    
	    \node (u5) at (6.5,0) {};
	    \node[label=below:\normalsize $u_{\frac{k}{2}+1}$] (u6) at (8,0) {};
	    \node[label=below:\normalsize $u_{\frac{k}{2}+2}$] (u7) at (9.5,0) {};
	    
	    \node (w1) at (6,1) {};
	    \node (w2) at (7,1) {};
	    \node (w3) at (8,1) {};
	    \pgftext[x=8cm,y=1.3cm] {\normalsize $v_{k-1}$};
	    \node[label=above:\normalsize $v_k$] (w4) at (9,1) {};
	    
	    \pgftext[x=-2cm,y=0cm] {\normalsize $G-V(M)$};
	    \pgftext[x=-2cm,y=1cm] {\normalsize $G(M)$};
	    
	    \foreach \from/\to in {w1/w2, w3/w4, u5/w2, u5/w3, u6/w1, u6/w4}
	    \draw [-] (\from) -- (\to);
	    
	    \foreach \from/\to in {w3/u6, u7/u6, u7/w3}
	    \draw [-, dashed] (\from) -- (\to);
	    
\end{tikzpicture}
\caption{The final situation in Case 4.3.3.} \label{figcase4.3.3}
\end{figure}
In view of Cases 4.1-4.3,
we may assume that $v_1$ and $v_2$ have degree $2$ in $G$.
Since $n(G)>4$,
we may assume that there is an edge $v_3v_4\in M$
distinct from $v_1v_2$ such that $u_2$ is adjacent to $v_3$.
Since $G$ is $2$-connected, 
$v_4$ has a neighbor $y$ distinct from $v_3$.
By Lemma \ref{lemma0}(ii) 
for the local pair $v_3u_2$ and $v_4y$,
$y=u_1$.
By symmetry between $v_1v_2$ and $v_3v_4$,
we obtain that $v_3$ and $v_4$ both have degree $2$ in $G$,
which implies $n(G)=6$.

\medskip

\noindent In view of the previous cases, 
we may assume that $G-V(M)$ consists of isolated vertices.

\medskip

\noindent {\bf Case 5.} {\it 
There are two vertices $u_1$ and $u_2$ in $V(G)\setminus V(M)$, 
and an edge $v_1v_2 \in M$ 
such that 
$u_1$ and $u_2$ are both adjacent to $v_1$ and $v_2$.}

\medskip

\noindent  Since $G$ is $2$-connected, and $n(G) > 4$,
$M$ contains an edge $v_3v_4$ 
such that $u_1$ is adjacent to $v_3$.
Since $G$ is $2$-connected,
$v_4$ has a neighbor $x$ distinct from $v_3$.
By Lemma \ref{lemma0}(ii), $x=u_2$.
By Lemma \ref{lemma0}(ii), 
$v_3$ and $v_4$ both have degree $2$ in $G$,
which implies $n(G)=6$.

\medskip

\noindent {\bf Case 6.} {\it $G$ contains a triangle.}

\medskip

\noindent In view of the previous cases,
$G$ contains a triangle $v_1v_2u_1v_1$
with $v_1v_2 \in M$.
Since $G$ is $2$-connected, and $n(G)>3$,
we may assume, by symmetry between $v_1$ and $v_2$,
that $v_2$ has a neighbor $u_2$ distinct from $u_1$ and $v_1$.
By the previous cases, 
and Lemma \ref{lemma0}(ii)
for the local pair $v_1u_1$ and $v_2u_2$, 
there is a $6$-cycle $v_1u_1v_3v_4u_2v_2v_1$ 
with $v_3v_4 \in M$.

First, assume that $v_1$ has a neighbor $x$ 
distinct from $u_1$ and $v_2$. 
By Case 5, $x \neq u_2$.
By Lemma \ref{lemma0}(ii)
for the local pair $v_1x$ and $v_2u_1$,
$x$ is adjacent to $v_4$. 
By Lemma \ref{lemma0}(ii)
for the local pair $v_1x$ and $v_2u_2$,
\begin{itemize}
\item either $x$ is adjacent to $v_3$, 
\item or $u_2$ is adjacent to $v_3$, 
\item or there is a $6$-cycle $v_1xv_5v_6u_2v_2v_1$, 
where $v_5v_6 \in M$ is distinct from $v_3v_4$. 
\end{itemize}
In the first two cases, we obtain $n(G)=7$,
hence, the third case applies.
If $v_3$ has a neighbor $z$ 
distinct from $u_1$ and $v_4$, 
then, 
by Lemma \ref{lemma0}(ii)
for the local pair $v_3z$ and $v_4u_2$
and
for the local pair $v_3z$ and $v_4x$,
$z$ is adjacent to $v_5$ and $v_6$, 
which implies $n(G) = 10$.
See the left part of Figure \ref{figcase6}.
Hence, $v_3$ has degree $2$ in $G$.
Since $G$ is $2$-connected, and $n(G)>9$,
$v_5$ has a neighbor $z$ distinct from $x$ and $v_6$.
Now,
Lemma \ref{lemma0}(ii) fails 
for the local pair $v_5z$ and $v_6u_2$.
Hence, $v_1$ has degree $2$ in $G$.

If $v_4$ has a neighbor $x$ distinct from $u_2$ and $v_3$,
then Lemma \ref{lemma0}(ii) fails 
for the local pair $v_3u_1$ and $v_4x$.
Hence, $v_4$ has degree $2$ in $G$.
Since $G$ is $2$-connected, and $n(G)>6$,
repeating the above arguments,
we obtain that 
$M$ contains an edge $v_5v_6$ 
distinct from $v_1v_2$ and $v_3v_4$,
and $V(G)\setminus V(M)$ contains a vertex $u_3$
such that 
$v_3$ is adjacent to $u_3$,
$u_3$ is adjacent to $v_6$,
$u_2$ is adjacent to $v_5$, 
and $v_6$ has degree $2$ in $G$.
Hence, for the vertices $v_5$, $v_6$, and $u_3$,
we are in a similar situation as for 
the vertices $v_3$, $v_4$, and $u_2$.
Setting up an inductive argument as in Case 4.3.3,
we obtain that $G$ is isomorphic to a graph in $\mathcal{B}_2$.  
See the right part of Figure \ref{figcase6}.

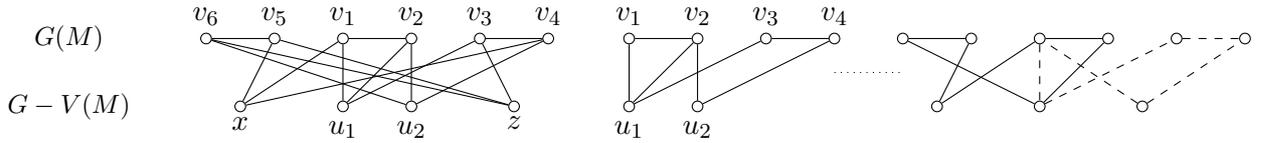
\begin{figure}[H]
\begin{minipage}{0.49\textwidth}
\centering\tiny
\begin{tikzpicture}[scale = 0.9] 
	    \node[label=below:\normalsize $u_1$] (u1) at (0,0) {};
	    \node[label=below:\normalsize $u_2$] (u2) at (1,0) {};
	    \node[label=below:\normalsize $x$] (x) at (-1.5,0) {};
	    \node[label=below:\normalsize $z$] (z) at (2.5,0) {};
	    \node[label=above:\normalsize $v_1$] (v1) at (0,1) {};
	    \node[label=above:\normalsize $v_2$] (v2) at (1,1) {};
	    \node[label=above:\normalsize $v_3$] (v3) at (2,1) {};
	    \node[label=above:\normalsize $v_4$] (v4) at (3,1) {};
	    \node[label=above:\normalsize $v_5$] (v5) at (-1,1) {};
	    \node[label=above:\normalsize $v_6$] (v6) at (-2,1) {};
	    
	    \foreach \from/\to in {v1/v2, v3/v4, v5/v6, x/v5, x/v1, x/v4, u1/v1, u1/v2, u1/v3, u2/v6, u2/v2, u2/v4, z/v3, z/v5, z/v6}
	    \draw [-] (\from) -- (\to);
	    
	    \pgftext[x=-4cm,y=0cm] {\normalsize $G-V(M)$};
	    \pgftext[x=-4cm,y=1cm] {\normalsize $G(M)$};
	    
\end{tikzpicture}
\end{minipage}
\begin{minipage}{0.49\textwidth}
\centering\tiny
\begin{tikzpicture}[scale = 0.9] 
	    \node[label=below:\normalsize $u_1$] (u1) at (0,0) {};
	    \node[label=below:\normalsize $u_2$] (u2) at (1,0) {};
	    \node[label=above:\normalsize $v_1$] (v1) at (0,1) {};
	    \node[label=above:\normalsize $v_2$] (v2) at (1,1) {};
	    \node[label=above:\normalsize $v_3$] (v3) at (2,1) {};
	    \node[label=above:\normalsize $v_4$] (v4) at (3,1) {};
	    
	    \foreach \from/\to in {v1/v2, v3/v4, u1/v2, u1/v1, u1/v3, u2/v2, u2/v4}
	    \draw [-] (\from) -- (\to);
	    
	    \draw[-, dotted] (3,0.5) -- (4,0.5);
	    
	    \node (u5) at (4.5,0) {};
	    \node (u6) at (6,0) {};
	    \node (u7) at (7.5,0) {};
	    
	    \node (w1) at (4,1) {};
	    \node (w2) at (5,1) {};
	    \node (w3) at (6,1) {};
	    \node (w4) at (7,1) {};
	    \node (w5) at (8,1) {};
	    \node (w6) at (9,1) {};

	    \foreach \from/\to in {w1/w2, w3/w4, u5/w2, u5/w3, u6/w1, u6/w4}
	    \draw [-] (\from) -- (\to);
	    
	    \foreach \from/\to in {w3/u6, u7/w3, u7/w6, w5/w6, u6/w5}
	    \draw [-, dashed] (\from) -- (\to);
	    
\end{tikzpicture}
\end{minipage}
\caption{Two situations in Case 6.}\label{figcase6}
\end{figure}

\noindent {\bf Case 7.} {\it $G$ contains a $4$-cycle.}

\medskip

\noindent In view of the previous cases,
there are two edges $v_1v_2$ and $v_3v_4$ in $M$,
and two vertices $u_1$ and $u_2$ in $V(G)\setminus V(M)$,
such that $u_1$ and $u_2$ 
are both adjacent to $v_1$ and $v_3$.
By Case 6, and symmetry between $v_2$ and $v_4$,
we may assume that $v_4$ has a neighbor $u_3$ distinct from $v_3$, $u_1$, and $u_2$.
By Lemma \ref{lemma0}(ii)
for the local pair $v_3u_1$ and $v_4u_3$,
\begin{itemize}
\item either $u_3$ is adjacent to $v_2$,
\item or there is a $6$-cycle $v_3u_1v_5v_6u_3v_4v_3$,
where $v_5v_6 \in M$ is distinct from $v_1v_2$. 
\end{itemize}
\noindent {\bf Case 7.1.} {\it $u_3$ is adjacent to $v_2$,
and $u_3$ has a neighbor $v_5$ distinct from $v_2$ and $v_4$.}

\medskip

\noindent $M$ contains an edge $v_5v_6$
distinct from $v_1v_2$ and $v_3v_4$.
If $v_6$ has a neighbor $y$ 
distinct from $v_5$, $u_1$, and $u_2$,
then Lemma \ref{lemma0}(ii) fails
for the local pair $v_5u_3$ and $v_6y$.
Hence, $v_6$ has no neighbor outside of $\{ u_1,u_2,v_5\}$.
Since $G$ is $2$-connected, 
we may assume, by symmetry between $u_1$ and $u_2$,
that $v_6$ is adjacent to $u_2$.
Since $G$ is $2$-connected, and $n(G)>9$,
at least one of the three vertices 
$v_2$, $v_4$, and $v_5$
has a neighbor $z$ in $V(G)\setminus V(M)$
that is distinct from $u_1$, $u_2$, and $u_3$.
By Lemma \ref{lemma0}(ii),
$z$ has at least two neighbors in $\{ v_2,v_4,v_5\}$.
Since $G$ is $2$-connected, and $n(G)>10$,
$z$ has exactly two neighbors in $\{ v_2,v_4,v_5\}$.

If $z$ has a neighbor $v_7$ 
such that $M$ contains an edge $v_7v_8$
distinct from $v_1v_2$, $v_3v_4$, and $v_5v_6$,
then, since $G$ is $2$-connected,
$v_8$ has a neighbor $z'$ distinct from $v_7$.
By Lemma \ref{lemma0}(ii),
$z'=u_1$, and $v_8$ has degree $2$ in $G$.
Since $G$ is $2$-connected, and $n(G)>12$,
$v_7$ has a neighbor $z''$ distinct from $z$ and $v_8$.
By Lemma \ref{lemma0}(ii)
for the local pair $v_7z''$ and $v_8u_1$,
$z''$ is adjacent to $v_2$ or $v_4$,
which implies that $n(G)=13$.
See the left part of Figure \ref{figcase7.1}.
Hence, $z$ has degree $2$ in $G$.

Since $G$ is $2$-connected, and $n(G) > 10$,
$M$ contains an edge $v_7v_8$
distinct from $v_1v_2$, $v_3v_4$, and $v_5v_6$
such that $u_1$ is adjacent to $v_7$.
Since $G$ is $2$-connected,
$v_8$ has a neighbor $z'$ distinct from $v_7$.
By Lemma \ref{lemma0}(ii)
for the local pair $v_7u_1$ and $v_8z'$,
$z'$ is adjacent to $v_2$ or $v_4$.
If $z'$ is adjacent to $v_2$,
then $z$ is adjacent to $v_4$ and $v_5$,
and Lemma \ref{lemma0}(ii) fails 
for the local pair $v_2z'$ and $v_1u_2$.
See the right part of Figure \ref{figcase7.1}.
If $z'$ is adjacent to $v_4$,
then $z$ is adjacent to $v_2$ and $v_5$,
and Lemma \ref{lemma0}(ii) fails 
for the local pair $v_4z'$ and $v_3u_2$.

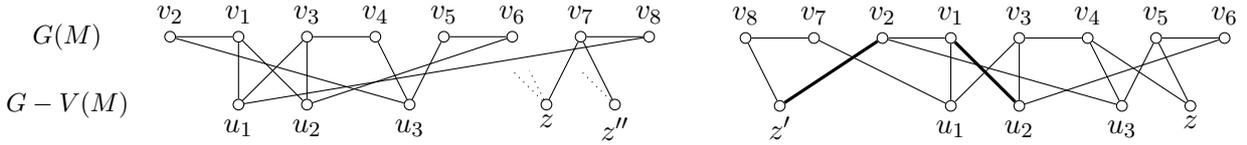
\begin{figure}[H]
\begin{minipage}{0.53\textwidth}
\centering\tiny
\begin{tikzpicture}[scale = 0.9] 
	    \node[label=below:\normalsize $u_1$] (u1) at (0,0) {};
	    \node[label=below:\normalsize $u_2$] (u2) at (1,0) {};
	    \node[label=below:\normalsize $u_3$] (u3) at (2.5,0) {};
	    \node[label=below:\normalsize $z$] (z) at (4.5,0) {};
	    \node[label=below:\normalsize $z^{\prime\prime}$] (z'') at (5.5,0) {};
	    \node[label=above:\normalsize $v_1$] (v1) at (0,1) {};
	    \node[label=above:\normalsize $v_2$] (v2) at (-1,1) {};
	    \node[label=above:\normalsize $v_3$] (v3) at (1,1) {};
	    \node[label=above:\normalsize $v_4$] (v4) at (2,1) {};
	    \node[label=above:\normalsize $v_5$] (v5) at (3,1) {};
	    \node[label=above:\normalsize $v_6$] (v6) at (4,1) {};
	    \node[label=above:\normalsize $v_7$] (v7) at (5,1) {};
	    \node[label=above:\normalsize $v_8$] (v8) at (6,1) {};
	    
	    \foreach \from/\to in {v1/v2, v3/v4, v5/v6, v7/v8, u1/v8, u1/v1, u1/v3, u2/v1, u2/v3, u2/v6, u3/v2, u3/v4, u3/v5, z/v7, z''/v7}
	    \draw [-] (\from) -- (\to);
	    
	    \draw[-,dotted] (z) -- (4,0.5);
	    \draw[-,dotted] (z) -- (4.25,0.5);
	    \draw[-,dotted] (z'') -- (5,0.5);
	    \pgftext[x=-2.5cm,y=0cm] {\normalsize $G-V(M)$};
	    \pgftext[x=-2.5cm,y=1cm] {\normalsize $G(M)$};
	    
\end{tikzpicture}
\end{minipage}
\begin{minipage}{0.46\textwidth}
\centering\tiny
\begin{tikzpicture}[scale = 0.9] 
	    \node[label=below:\normalsize $u_1$] (u1) at (0,0) {};
	    \node[label=below:\normalsize $u_2$] (u2) at (1,0) {};
	    \node[label=below:\normalsize $u_3$] (u3) at (2.5,0) {};
	    \node[label=below:\normalsize $z$] (z) at (3.5,0) {};
	    \node[label=below:\normalsize $z^{\prime}$] (z') at (-2.5,0) {};
	    \node[label=above:\normalsize $v_1$] (v1) at (0,1) {};
	    \node[label=above:\normalsize $v_2$] (v2) at (-1,1) {};
	    \node[label=above:\normalsize $v_3$] (v3) at (1,1) {};
	    \node[label=above:\normalsize $v_4$] (v4) at (2,1) {};
	    \node[label=above:\normalsize $v_5$] (v5) at (3,1) {};
	    \node[label=above:\normalsize $v_6$] (v6) at (4,1) {};
	    \node[label=above:\normalsize $v_7$] (v7) at (-2,1) {};
	    \node[label=above:\normalsize $v_8$] (v8) at (-3,1) {};
	    
	    \foreach \from/\to in {v1/v2, v3/v4, v5/v6, v7/v8, u1/v7, u1/v1, u1/v3, u2/v3, u2/v6, u3/v2, u3/v4, u3/v5, z/v4, z/v5, z'/v8}
	    \draw [-] (\from) -- (\to);
	    
	    \draw[-,very thick] (v2) -- (z');
	    \draw[-,very thick] (v1) -- (u2);

\end{tikzpicture}
\end{minipage}
\caption{Two situations in Case 7.1.}\label{figcase7.1}
\label{fig13}
\end{figure}

\noindent {\bf Case 7.2.} {\it $u_3$ is adjacent to $v_2$,
and $u_3$ has degree $2$ in $G$.}

\medskip

\noindent First, we assume that $v_2$ has a neighbor $u_4$ 
distinct from $v_1$ and $u_3$.
If $u_4$ is adjacent to $v_4$, 
then, by symmetry between $u_3$ and $u_4$,
$u_4$ has degree $2$ in $G$.
Since $G$ is $2$-connected, and $n(G)>8$,
$M$ contains an edge $v_5v_6$ 
distinct from $v_1v_2$ and $v_3v_4$
such that $u_1$ is adjacent to $v_5$.
Since $G$ is $2$-connected, 
$v_6$ has a neighbor $x$ distinct from $v_5$,
and Lemma \ref{lemma0}(ii) fails
for the local pair $v_5u_1$ and $v_6x$.
Hence, $u_4$ is not adjacent to $v_4$.
By Lemma \ref{lemma0}(ii)
for the local pair $v_1u_1$ and $v_2u_4$,
there is a $6$-cycle $v_2u_4v_5v_6u_1v_1v_2$, 
where $v_5v_6 \in M$ is distinct from $v_3v_4$.
By Lemma \ref{lemma0}(ii)
for the local pair $v_1u_2$ and $v_2u_4$,
\begin{itemize}
\item either $u_2$ is adjacent to $v_6$,
\item or there is a $6$-cycle $v_1u_2v_7v_8u_4v_2v_1$, where $v_7v_8 \in M$ is distinct from
$v_3v_4$ and $v_5v_6$.
\end{itemize}
First, we assume that $u_2$ is adjacent to $v_6$.
Since $G$ is $2$-connected, and $n(G)>10$,
$v_4$ or $v_5$ has a neighbor $x$ distinct from 
$v_3$,
$u_3$,
$v_6$, and
$u_4$.
By Lemma \ref{lemma0}(ii),
$x$ is adjacent to $v_4$ and $v_5$.
Since $G$ is $2$-connected, and $n(G)>11$,
$M$ contains an edge $v_7v_8$
such that $x$ is adjacent to $v_7$.
Since $G$ is $2$-connected, 
$v_8$ has a neighbor $y$ distinct from $v_7$.
Now, Lemma \ref{lemma0}(ii) fails
for the local pair $v_7x$ and $v_8y$.
See the left part of Figure \ref{figcase7.2}.
Hence, $u_2$ is not adjacent to $v_6$, 
and there is a $6$-cycle $v_1u_2v_7v_8u_4v_2v_1$, 
where $v_7v_8 \in M$ is distinct from $v_3v_4$ and $v_5v_6$.

If $v_4$ has a neighbor $x$ distinct from $v_3$ and $u_3$,
then, by Lemma \ref{lemma0}(ii),
$x$ is adjacent to $v_5$ and $v_8$. 
Since $G$ is $2$-connected, and $n(G) > 13$, 
$v_6$ or $v_7$ has neighbor $y$
distinct from 
$u_1$,
$v_5$,
$u_2$, and 
$v_8$.
By Lemma \ref{lemma0}(ii),
$y$ is adjacent to $v_6$ and $v_7$,
and, since $G$ is $2$-connected, $n(G)=14$.
See the right part of Figure \ref{figcase7.2}.
Hence, $v_4$ has degree $2$ in $G$.

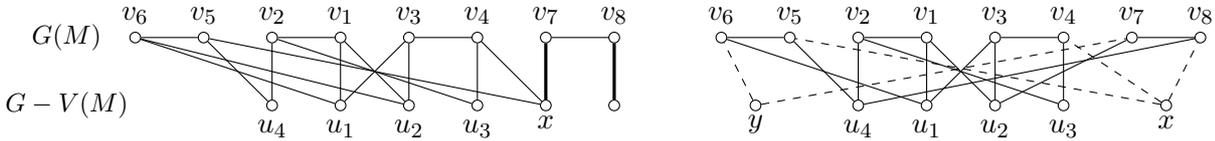
\begin{figure}[H]
\begin{minipage}{0.49\textwidth}
\centering\tiny
\begin{tikzpicture}[scale = 0.9] 
	    \node[label=below:\normalsize $u_1$] (u1) at (0,0) {};
	    \node[label=below:\normalsize $u_2$] (u2) at (1,0) {};
	    \node[label=below:\normalsize $u_3$] (u3) at (2,0) {};
	    \node[label=below:\normalsize $u_4$] (u4) at (-1,0) {};
	    \node[label=below:\normalsize $x$] (x) at (3,0) {};
	    \node (y) at (4,0) {};
	    \node[label=above:\normalsize $v_1$] (v1) at (0,1) {};
	    \node[label=above:\normalsize $v_2$] (v2) at (-1,1) {};
	    \node[label=above:\normalsize $v_3$] (v3) at (1,1) {};
	    \node[label=above:\normalsize $v_4$] (v4) at (2,1) {};
	    \node[label=above:\normalsize $v_5$] (v5) at (-2,1) {};
	    \node[label=above:\normalsize $v_6$] (v6) at (-3,1) {};
	    \node[label=above:\normalsize $v_7$] (v7) at (3,1) {};
	    \node[label=above:\normalsize $v_8$] (v8) at (4,1) {};
	    
	    \foreach \from/\to in {v1/v2, v3/v4, v5/v6, v7/v8, u1/v6, u1/v1, u1/v3, u2/v1, u2/v3, u2/v6, u3/v2, u3/v4, u4/v5, u4/v2, x/v4, x/v5}
	    \draw [-] (\from) -- (\to);
	    
	    \draw[-,very thick] (v8) -- (y);
	    \draw[-,very thick] (v7) -- (x);
	    \pgftext[x=-4cm,y=0cm] {\normalsize $G-V(M)$};
	    \pgftext[x=-4cm,y=1cm] {\normalsize $G(M)$};
	    
\end{tikzpicture}
\end{minipage}
\begin{minipage}{0.49\textwidth}
\centering\tiny
\begin{tikzpicture}[scale = 0.9] 
   \node[label=below:\normalsize $u_1$] (u1) at (0,0) {};
	    \node[label=below:\normalsize $u_2$] (u2) at (1,0) {};
	    \node[label=below:\normalsize $u_3$] (u3) at (2,0) {};
	    \node[label=below:\normalsize $u_4$] (u4) at (-1,0) {};
	    \node[label=below:\normalsize $x$] (x) at (3.5,0) {};
	    \node[label=below:\normalsize $y$] (y) at (-2.5,0) {};
	    \node[label=above:\normalsize $v_1$] (v1) at (0,1) {};
	    \node[label=above:\normalsize $v_2$] (v2) at (-1,1) {};
	    \node[label=above:\normalsize $v_3$] (v3) at (1,1) {};
	    \node[label=above:\normalsize $v_4$] (v4) at (2,1) {};
	    \node[label=above:\normalsize $v_5$] (v5) at (-2,1) {};
	    \node[label=above:\normalsize $v_6$] (v6) at (-3,1) {};
	    \node[label=above:\normalsize $v_7$] (v7) at (3,1) {};
	    \node[label=above:\normalsize $v_8$] (v8) at (4,1) {};
	    
	    \foreach \from/\to in {v1/v2, v3/v4, v5/v6, v7/v8, u1/v6, u1/v1, u1/v3, u2/v1, u2/v3, u2/v7, u3/v2, u3/v4, u4/v5, u4/v2, u4/v8}
	    \draw [-] (\from) -- (\to);
	    
	    \foreach \from/\to in {x/v4, x/v8, x/v5, y/v7, y/v6}
	    \draw [-, dashed] (\from) -- (\to);

\end{tikzpicture}
\end{minipage}
\caption{Two situations in Case 7.2.}\label{figcase7.2}
\end{figure}

\noindent By Lemma \ref{lemma0}(ii),
$v_5$ and $v_8$ have degree $2$ in $G$.
Since $G$ is $2$-connected, and $n(G) > 12$, 
$v_6$ or $v_7$ has neighbor $x$
distinct from 
$u_1$,
$v_5$,
$u_2$, and 
$v_8$.
By Lemma \ref{lemma0}(ii),
$x$ is adjacent to $v_6$ and $v_7$,
and, since $G$ is $2$-connected, $n(G)=13$.
Hence, $v_2$ has degree $2$ in $G$.

By symmetry between $v_2$ and $v_4$, 
we may assume that $v_4$ has degree $2$ in $G$.
Since $G$ is $2$-connected, and $n(G)>7$, 
$M$ contains an edge $v_5v_6$ distinct from $v_1v_2$ and $v_3v_4$
such that $u_1$ is adjacent to $v_5$.
Since $G$ is $2$-connected,
$v_6$ has a neighbor $x$ distinct from $v_5$.
Now, Lemma \ref{lemma0}(ii) fails
for the local pair $v_5u_1$ and $v_6x$.

\medskip

\noindent {\bf Case 7.3.} {\it $u_3$ is not adjacent to $v_2$.}

\medskip

\noindent As observed above,
there is a $6$-cycle $v_3u_1v_5v_6u_3v_4v_3$,
where $v_5v_6 \in M$ is distinct from $v_1v_2$.
Since $G$ is $2$-connected, 
$v_2$ has a neighbor $u_4$ 
distinct from $v_1$ and $u_3$.
By symmetry between $u_3$ and $u_4$, 
we may assume that $u_4$ is not adjacent to $v_4$.
By Lemma \ref{lemma0}(ii) 
for the local pair $v_1u_1$ and $v_2u_4$,
$u_4$ is adjacent to $v_6$.

First, we assume that $u_2$ is adjacent to $v_5$. 
If $M$ contains an edge $v_7v_8$ 
distinct from $v_1v_2$, $v_3v_4$, and $v_5v_6$
such that $u_3$ is adjacent to $v_7$,
then $v_8$ has a neighbor $x$ distinct from $v_7$,
and Lemma \ref{lemma0}(ii) fails
for the local pair $v_7u_3$ and $v_8x$.
Hence, $u_3$ has degree $2$ in $G$.
Similarly, it follows that 
$u_4$ has degree $2$ in $G$.
Since $G$ is $2$-connected, and $n(G)>10$,
$v_2$ or $v_4$ has a neighbor $x$.
By Lemma \ref{lemma0}(ii),
$x$ is adjacent to $v_2$ and $v_4$,
and $n(G)=11$.
See the left part of Figure \ref{figcase7.3}.
Hence, $u_2$ is not adjacent to $v_5$.

By Lemma \ref{lemma0}(ii) 
for the local pair $v_1u_2$ and $v_2u_4$,
there is a $6$-cycle $v_1u_2v_7v_8u_4v_2v_1$,
where $v_7v_8 \in M$ is distinct from $v_3v_4$ and $v_5v_6$.
By Lemma \ref{lemma0}(ii) 
for the local pair $v_3u_2$ and $v_4u_3$,
$u_3$ is adjacent to $v_8$.
If $v_2$ or $v_4$ have a neighbor $x$ 
distinct from 
$v_1$,
$u_4$,
$v_3$, and
$u_3$,
then, by Lemma \ref{lemma0}(ii),
$x$ is adjacent to $v_2$ and $v_4$.
Similarly, 
if $v_5$ or $v_7$ have a neighbor $y$ 
distinct from 
$v_6$,
$u_1$,
$v_8$, and
$u_2$,
then, by Lemma \ref{lemma0}(ii),
$y$ is adjacent to $v_5$ and $v_7$.
Since, by Lemma \ref{lemma0}(ii),
$x$ and $y$ have degree $2$ in $G$, 
if they exist,
we obtain $n(G)\leq 14$.
See the right part of Figure \ref{figcase7.3}.

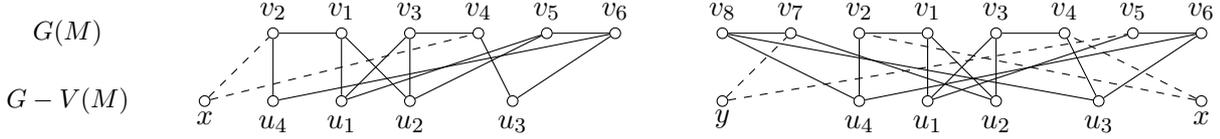
\begin{figure}[H]
\begin{minipage}{0.49\textwidth}
\centering\tiny
\begin{tikzpicture}[scale = 0.9] 
	    \node[label=below:\normalsize $u_1$] (u1) at (0,0) {};
	    \node[label=below:\normalsize $u_2$] (u2) at (1,0) {};
	    \node[label=below:\normalsize $u_3$] (u3) at (2.5,0) {};
	    \node[label=below:\normalsize $u_4$] (u4) at (-1,0) {};
	    \node[label=below:\normalsize $x$] (x) at (-2,0) {};
	    \node[label=above:\normalsize $v_1$] (v1) at (0,1) {};
	    \node[label=above:\normalsize $v_2$] (v2) at (-1,1) {};
	    \node[label=above:\normalsize $v_3$] (v3) at (1,1) {};
	    \node[label=above:\normalsize $v_4$] (v4) at (2,1) {};
	    \node[label=above:\normalsize $v_5$] (v5) at (3,1) {};
	    \node[label=above:\normalsize $v_6$] (v6) at (4,1) {};
	    
	    \foreach \from/\to in {v1/v2, v3/v4, v5/v6, u1/v5, u1/v1, u1/v3, u2/v1, u2/v3, u2/v5, u3/v4, u3/v6, u4/v2, u4/v6}
	    \draw [-] (\from) -- (\to);
	    
	    \draw[-,dashed] (x)--(v2);
	    \draw[-,dashed] (x)--(v4);
	    
	    \pgftext[x=-4cm,y=0cm] {\normalsize $G-V(M)$};
	    \pgftext[x=-4cm,y=1cm] {\normalsize $G(M)$};
	    
\end{tikzpicture}
\end{minipage}
\begin{minipage}{0.49\textwidth}
\centering\tiny
\begin{tikzpicture}[scale = 0.9] 
	    \node[label=below:\normalsize $u_1$] (u1) at (0,0) {};
	    \node[label=below:\normalsize $u_2$] (u2) at (1,0) {};
	    \node[label=below:\normalsize $u_3$] (u3) at (2.5,0) {};
	    \node[label=below:\normalsize $u_4$] (u4) at (-1,0) {};
	    \node[label=below:\normalsize $x$] (x) at (4,0) {};
	    \node[label=below:\normalsize $y$] (y) at (-3,0) {};
	    \node[label=above:\normalsize $v_1$] (v1) at (0,1) {};
	    \node[label=above:\normalsize $v_2$] (v2) at (-1,1) {};
	    \node[label=above:\normalsize $v_3$] (v3) at (1,1) {};
	    \node[label=above:\normalsize $v_4$] (v4) at (2,1) {};
	    \node[label=above:\normalsize $v_5$] (v5) at (3,1) {};
	    \node[label=above:\normalsize $v_6$] (v6) at (4,1) {};
	    \node[label=above:\normalsize $v_7$] (v7) at (-2,1) {};
	    \node[label=above:\normalsize $v_8$] (v8) at (-3,1) {};
	    
	    \foreach \from/\to in {v1/v2, v3/v4, v5/v6, v7/v8, u1/v5, u1/v1, u1/v3, u2/v1, u2/v3, u2/v7, u3/v4, u3/v6, u3/v8, u4/v2, u4/v6, u4/v8}
	    \draw [-] (\from) -- (\to);
	    
	    \foreach \from/\to in {x/v2, x/v4, y/v5, y/v7}
	    \draw [-,dashed] (\from) -- (\to);

\end{tikzpicture}
\end{minipage}
\caption{Two situations in Case 7.3.}\label{figcase7.3}
\end{figure}

\medskip

\noindent {\bf Case 8.} {\it $M$ contains an edge $v_1v_2$,
where $v_1$ and $v_2$ have degree $2$ in $G$.}

\medskip

\noindent Let $u_1$ be the neighbor of $v_1$ distinct from $v_2$,
and let $u_2$ be the neighbor of $v_2$ distinct from $v_1$.
By Lemma \ref{lemma0}(ii) 
for the local pair $v_1u_1$ and $v_2u_2$,
there is a $6$-cycle $v_1u_1v_3v_4u_2v_2v_1$,
where $v_3v_4\in M$.

If $v_3$ and $v_4$ have degree $2$ in $G$, then, 
since $G$ is $2$-connected, and $n(G)>6$,
$M$ contains an edge $v_5v_6$ distinct from $v_1v_2$ and $v_3v_4$
such that $u_1$ is adjacent to $v_5$.
Since $G$ is $2$-connected, 
$v_6$ has a neighbor $x$ distinct from $v_5$.
By Lemma \ref{lemma0}(ii) 
for the local pair $v_5u_1$ and $v_6x$,
$x=u_2$, and $n(G)=8$.
Hence, by symmetry between $v_3$ and $v_4$,
we may assume that $v_4$ has a neighbor $u_3$ 
distinct from $u_2$ and $v_3$.
By Lemma \ref{lemma0}(ii) 
for the local pair $v_3u_1$ and $v_4u_3$,
there is a $6$-cycle $v_3u_1v_5v_6u_3v_4v_3$,
where $v_5v_6\in M$ is distinct from $v_1v_2$.
In view of Case 7, 
we may assume that $v_6$ is not adjacent to $u_2$.
By Lemma \ref{lemma0}(ii), 
$v_6$ has degree $2$ in $G$.

First, we assume that $u_2$ has a neighbor $v_7$
distinct from $v_2$ and $v_4$.
By Case 6 and Lemma \ref{lemma0}(ii), 
$v_7$ is distinct from $v_3$ and $v_5$.
Let $M$ contain the edge $v_7v_8$.
Since $G$ is $2$-connected,
$v_8$ has a neighbor $x$ distinct from $v_7$.
By Lemma \ref{lemma0}(ii) 
for the local pair $v_7u_2$ and $v_8x$,
$x$ is adjacent to $v_3$.
This implies that $x$ is distinct from $u_3$,
and that $v_8$ has degree $2$ in $G$.
By Case 7, 
$u_3$ is not adjacent to $v_7$,
and
$x$ is not adjacent to $v_5$.
By Lemma \ref{lemma0}(ii) 
for the local pair $v_3x$ and $v_4u_3$,
there is a $6$-cycle $v_3xv_9v_{10}u_3v_4v_3$,
where $v_9v_{10}\in M$ is distinct from 
$v_1v_2$,
$v_5v_6$, and
$v_7v_8$.
If $v_5$ or $v_9$ has a neighbor $y$
distinct from
$u_1$,
$v_6$,
$x$, and 
$v_{10}$,
then, by Lemma \ref{lemma0}(ii),
$y$ is adjacent to $v_5$ and $v_9$.
Similarly, 
if $v_7$ or $v_{10}$ has a neighbor $z$
distinct from
$u_2$,
$v_8$,
$u_3$, and 
$v_9$,
then, by Lemma \ref{lemma0}(ii),
$z$ is adjacent to $v_7$ and $v_{10}$.
If $y$ and $z$ both exist,
then, by Lemma \ref{lemma0}(ii) 
for the local pair $v_9y$ and $v_{10}z$,
there is a $6$-cycle $v_9yv_{11}v_{12}zv_{10}v_9$,
where $v_{11}v_{12}\in M$ is distinct from 
$v_1v_2$,
$v_3v_4$,
$v_5v_6$, and
$v_7v_8$.
By Lemma \ref{lemma0}(ii),
$v_{11}$ and $v_{12}$ have degree $2$ in $G$,
which implies $n(G)=18$.
See Figure \ref{figcase8}.
If at most one of $y$ or $z$ exists,
then $n(G)\leq 15$.

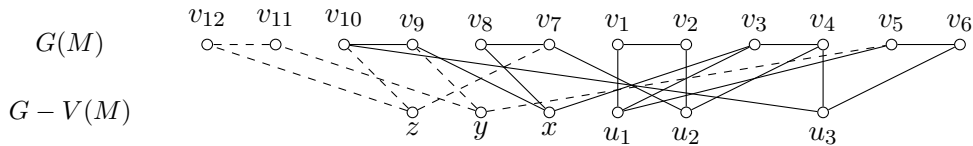
\begin{figure}[H]
\centering\tiny
\begin{tikzpicture}[scale = 0.9] 
	    \node[label=below:\normalsize $u_1$] (u1) at (0,0) {};
	    \node[label=below:\normalsize $u_2$] (u2) at (1,0) {};
	    \node[label=below:\normalsize $u_3$] (u3) at (3,0) {};
	    \node[label=below:\normalsize $x$] (x) at (-1,0) {};
	    \node[label=below:\normalsize $y$] (y) at (-2,0) {};
	    \node[label=below:\normalsize $z$] (z) at (-3,0) {};
	    \node[label=above:\normalsize $v_1$] (v1) at (0,1) {};
	    \node[label=above:\normalsize $v_2$] (v2) at (1,1) {};
	    \node[label=above:\normalsize $v_3$] (v3) at (2,1) {};
	    \node[label=above:\normalsize $v_4$] (v4) at (3,1) {};
	    \node[label=above:\normalsize $v_5$] (v5) at (4,1) {};
	    \node[label=above:\normalsize $v_6$] (v6) at (5,1) {};
	    \node[label=above:\normalsize $v_7$] (v7) at (-1,1) {};
	    \node[label=above:\normalsize $v_8$] (v8) at (-2,1) {};
	    \node[label=above:\normalsize $v_9$] (v9) at (-3,1) {};
	    \node[label=above:\normalsize $v_{10}$] (v10) at (-4,1) {};
	    \node[label=above:\normalsize $v_{11}$] (v11) at (-5,1) {};
	    \node[label=above:\normalsize $v_{12}$] (v12) at (-6,1) {};
	    
	    \foreach \from/\to in {v1/v2, v3/v4, v5/v6, v7/v8, v9/v10, u1/v5, u1/v1, u1/v3, u2/v2, u2/v4, u2/v7, u3/v4, u3/v6, u3/v10, x/v9, x/v8, x/v3}
	    \draw [-] (\from) -- (\to);
	    
	    \foreach \from/\to in {v11/v12, z/v10, z/v7, z/v12, y/v9, y/v5, y/v11}
	    \draw [-,dashed] (\from) -- (\to);
	    
	    \pgftext[x=-8cm,y=0cm] {\normalsize $G-V(M)$};
	    \pgftext[x=-8cm,y=1cm] {\normalsize $G(M)$};
\end{tikzpicture}
\caption{A situation in Case 8.}\label{figcase8}
\end{figure}

\noindent Hence, $u_2$ has degree $2$ in $G$.
By Case 6, and Lemma \ref{lemma0}(ii),
$v_3$ has degree $2$ in $G$.
Setting up an inductive argument as in Case 4.3.3 and Case 6,
it follows that $G$ is isomorphic to a graph in ${\cal B}_1$.
See Figure \ref{figcase8b}.

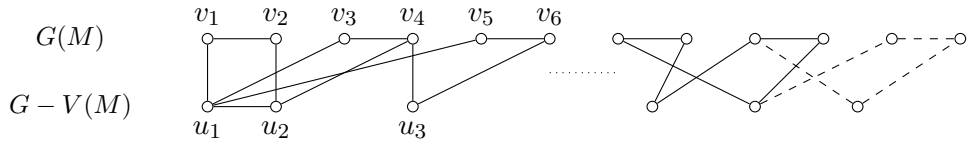
\begin{figure}[H]
\centering\tiny
\begin{tikzpicture}[scale = 0.9] 
	    \node[label=below:\normalsize $u_1$] (u1) at (0,0) {};
	    \node[label=below:\normalsize $u_2$] (u2) at (1,0) {};
	    \node[label=below:\normalsize $u_3$] (u3) at (3,0) {};
	    \node[label=above:\normalsize $v_1$] (v1) at (0,1) {};
	    \node[label=above:\normalsize $v_2$] (v2) at (1,1) {};
	    \node[label=above:\normalsize $v_3$] (v3) at (2,1) {};
	    \node[label=above:\normalsize $v_4$] (v4) at (3,1) {};
	    \node[label=above:\normalsize $v_5$] (v5) at (4,1) {};
	    \node[label=above:\normalsize $v_6$] (v6) at (5,1) {};
	    
	    \foreach \from/\to in {v1/v2, v3/v4, v5/v6, u1/u2, u1/v1, u1/v3, u1/v5, u2/v2, u2/v4, u3/v4, u3/v6}
	    \draw [-] (\from) -- (\to);
	    
	    \draw[-, dotted] (5,0.5) -- (6,0.5);
	    
	    \node (u5) at (6.5,0) {};
	    \node (u6) at (8,0) {};
	    \node (u7) at (9.5,0) {};
	    
	    \node (w1) at (6,1) {};
	    \node (w2) at (7,1) {};
	    \node (w3) at (8,1) {};
	    \node (w4) at (9,1) {};
	    \node (w5) at (10,1) {};
	    \node (w6) at (11,1) {};
	    
	    \pgftext[x=-2cm,y=0cm] {\normalsize $G-V(M)$};
	    \pgftext[x=-2cm,y=1cm] {\normalsize $G(M)$};
	    
	    \foreach \from/\to in {w1/w2, w3/w4, u5/w2, u5/w3, u6/w1, u6/w4}
	    \draw [-] (\from) -- (\to);
	    
	    \foreach \from/\to in {u7/w3, u7/w6, w5/w6, u6/w5}
	    \draw [-, dashed] (\from) -- (\to);
\end{tikzpicture}
\caption{The final situation in Case 8.}\label{figcase8b}
\end{figure}

\medskip

\noindent {\bf Case 9.} {\it $M$ contains an edge $v_1v_2$,
where $v_1$ and $v_2$ have degree $3$ in $G$.}

\medskip

\noindent Let $u_1$ and $u_2$ be the neighbors of $v_1$ distinct from $v_2$,
and 
let $u_3$ and $u_4$ be the neighbors of $v_2$ distinct from $v_1$.
By Lemma \ref{lemma0}(ii) 
for the local pair $v_1u_1$ and $v_2u_4$,
there is a $6$-cycle $v_1u_1v_3v_4u_4v_2v_1$,
where $v_3v_4\in M$.
By Lemma \ref{lemma0}(ii) 
for the local pair $v_1u_1$ and $v_2u_3$,
there is a $6$-cycle $v_1u_1v_5v_6u_3v_2v_1$,
where $v_5v_6\in M$.
By Case 7, 
the edges $v_3v_4$ and $v_5v_6$ are distinct.
By Lemma \ref{lemma0}(ii) 
for the local pair $v_1u_2$ and $v_2u_4$,
there is a $6$-cycle $v_1u_2v_7v_8u_4v_2v_1$,
where $v_7v_8\in M$.
By Case 7, 
the edges $v_3v_4$ and $v_7v_8$ are distinct.
If the edge $v_7v_8$ equals the edge $v_5v_6$, 
then $v_7=v_6$ and $v_8=v_5$. 
By Lemma \ref{lemma0}(ii) 
for the local pair $v_5u_4$ and $v_6u_3$,
$u_3$ is adjacent to $v_3$.
By Lemma \ref{lemma0}(ii) 
for the local pair $v_5u_1$ and $v_6u_2$,
$u_2$ is adjacent to $v_4$,
and $n(G)=10$.
Hence, the edges $v_7v_8$ and $v_5v_6$ are distinct.
By Lemma \ref{lemma0}(ii) 
for the local pair $v_1u_2$ and $v_2u_3$,
there is a $6$-cycle $v_1u_2v_9v_{10}u_3v_2v_1$,
where $v_{9}v_{10} \in M$.
By Case 7, 
$v_{9}v_{10}$ is distinct from $v_5v_6$ and $v_7v_8$.
If $v_9v_{10} = v_3v_4$, then $v_9 = v_4$ and $v_{10} = v_3$,
and Lemma \ref{lemma0}(ii) fails
for the local pair $v_3u_1$ and $v_4u_2$.
Hence, the edges $v_9v_{10}$ and $v_3v_4$ are distinct.

Let $U= \{(8,10), (5,9), (4,6), (3,7)\}$.
By Lemma \ref{lemma0}(ii),
for every $(i,j) \in U$,
if $v_i$ or $v_j$ has a neighbor $x$
not in $\{ u_1,u_2,u_3,u_4\}$,
then $x$ is adjacent to $v_i$ and $v_j$.
By Case 8, 
the set $Y=N_G(\{ v_1,\ldots,v_{10}\}) \setminus \{u_1,u_2,u_3,u_4 \}$ 
contains at least two vertices. 
Note that every vertex in $Y$ has exactly two neighbors in $\{ v_1,\ldots,v_{10}\}$,
which implies $|Y|\leq 4$.
If some vertex $y\in Y$ has a neighbor $w_1$
not in $\{ v_1,\ldots,v_{10}\}$,
then $M$ contains an edge $w_1w_2$,
and, by Lemma \ref{lemma0}(ii),
$w_1$ and $w_2$ have all their neighbors in $Y\cup \{w_1,w_2\}$.
By Case 8,
there are at most $2$ vertices in $V(M)\setminus \{ v_1,\ldots,v_{10}\}$
that have neighbors in $Y$, which implies
that $n(G)\leq 10+4+4+2=20$.
See Figure \ref{figcase9}.

\begin{figure}[H]
\centering\tiny
\begin{tikzpicture}[scale = 0.9] 
	    \node[label=below:\normalsize $u_1$] (u1) at (-1,0) {};
	    \node[label=below:\normalsize $u_2$] (u2) at (0,0) {};
	    \node[label=below:\normalsize $u_3$] (u3) at (1,0) {};
	    \node[label=below:\normalsize $u_4$] (u4) at (2,0) {};
	    \node (y1) at (4,-0.5) {};
	    \node (y2) at (5,-0.5) {};
	    \node (y3) at (6,-0.5) {};
	    \node (y4) at (7,-0.5) {};
	    \node[label=above:\normalsize $v_1$] (v1) at (0,1) {};
	    \node[label=above:\normalsize $v_2$] (v2) at (1,1) {};
	    \node[label=above:\normalsize $v_3$] (v3) at (2,1) {};
	    \node[label=above:\normalsize $v_4$] (v4) at (3,1) {};
	    \node[label=above:\normalsize $v_5$] (v5) at (-2,1) {};
	    \node[label=above:\normalsize $v_6$] (v6) at (-1,1) {};
	    \node[label=above:\normalsize $v_7$] (v7) at (4,1) {};
	    \node[label=above:\normalsize $v_8$] (v8) at (5,1) {};
	    \node[label=above:\normalsize $v_9$] (v9) at (-4,1) {};
	    \node[label=above:\normalsize $v_{10}$] (v10) at (-3,1) {};
	    \node[label=above:\normalsize $w_1$] (w1) at (7,1) {};
	    \node[label=above:\normalsize $w_2$] (w2) at (8,1) {};
	    
	    \foreach \from/\to in {v1/v2, v3/v4, v5/v6, v7/v8, v9/v10, u1/v5, u1/v1, u1/v3, u2/v1, u2/v7, u2/v9, u3/v2, u3/v6, u3/v10, u4/v2, u4/v4, u4/v8}
	    \draw [-] (\from) -- (\to);
	    
	   \draw[-,dashed] (w1) -- (w2);
	    
	    \draw[-,dashed] (w1) -- (6,0.25);
	    \draw[-,dashed] (w1) -- (6.5,0.25);
	    \draw[-,dashed] (w2) -- (7,0.25);
	    \draw[-,dashed] (w2) -- (7.5,0.25);

	    \draw[-,dashed] (y1) -- (3,0.2);
	    \draw[-,dashed] (y1) -- (3.25,0.2);
	    \draw[-,dashed] (y2) -- (3.75,0.2);
	    \draw[-,dashed] (y2) -- (4,0.2);
	    \draw[-,dashed] (y3) -- (4.5,0.2);
	    \draw[-,dashed] (y3) -- (4.75,0.2);
	    \draw[-,dashed] (y4) -- (5.25,0.2);
	    \draw[-,dashed] (y4) -- (5.5,0.2);
	    
	    \pgftext[x=-6cm,y=0cm] {\normalsize $G-V(M)$};
	    \pgftext[x=-6cm,y=1cm] {\normalsize $G(M)$};
	    
	    \pgftext[x=8cm,y=-0.5cm] {\normalsize $Y$};
	    
	    \draw (5.5,-0.5) ellipse (2cm and 0.3cm);
\end{tikzpicture}
\caption{The final situation in Case 9.}\label{figcase9}
\end{figure}
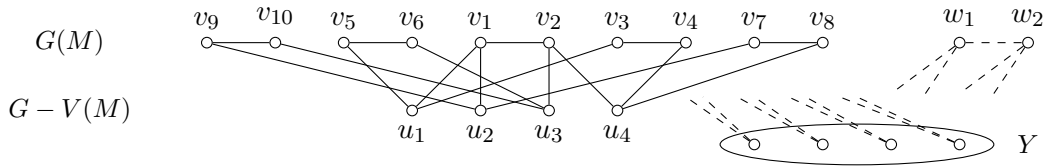

\medskip

\noindent {\bf Case 10.} {\it $M$ contains an edge $v_1v_2$,
where $v_1$ has degree $2$, and $v_2$ have degree $3$ in $G$.}

\medskip

\noindent Let $u_1$ be the neighbor of $v_1$ distinct from $v_2$,
and
let $u_2$ and $u_3$ be the neighbors of $v_2$ distinct from $v_1$.
By Case 6, $u_1$, $u_2$, and $u_3$ are distinct.
By Lemma \ref{lemma0}(ii) 
for the local pair $v_1u_1$ and $v_2u_2$,
there is a $6$-cycle $v_1u_1v_3v_4u_2v_2v_1$,
where $v_3v_4 \in M$.
By Lemma \ref{lemma0}(ii) 
for the local pair $v_1u_1$ and $v_2u_3$,
there is a $6$-cycle $v_1u_1v_5v_6u_3v_2v_1$,
where $v_5v_6 \in M$.
By Case 6 and Case 7,
the edges $v_3v_4$ and $v_5v_6$ are distinct,
$v_4$ is not adjacent to $u_3$, and
$v_6$ is not adjacent to $u_2$.

First, we assume that $v_4$ or $v_6$ has a neighbor
distinct from
$v_3$,
$u_2$,
$v_5$, and
$u_3$.
By Lemma \ref{lemma0}(ii),
$x$ is adjacent to $v_4$ and $v_6$.
By Case 9, 
$v_3$ and $v_5$ have degree $2$ in $G$.
If $x$ has a neighbor $v_7$ distinct from $v_4$ and $v_6$,
then $M$ contains an edge $v_7v_8$
distinct from $v_1v_2$, $v_3v_4$, and $v_5v_6$.
Since $G$ is $2$-connected, $v_8$ has a neighbor $y$ distinct from $v_7$,
and Lemma \ref{lemma0}(ii) fails 
for the local pair $v_7x$ and $v_8y$.
Hence, $x$ has degree $2$ in $G$.
Since $G$ is $2$-connected, and $n(G)>10$,
$u_3$ has a neighbor $v_7$,
and 
$M$ contains an edge $v_7v_8$
distinct from $v_1v_2$, $v_3v_4$, and $v_5v_6$.
Since $G$ is $2$-connected, $v_8$ has a neighbor $y$ distinct from $v_7$,
and Lemma \ref{lemma0}(ii) fails 
for the local pair $v_7u_3$ and $v_8y$.
Hence, $v_4$ and $v_6$ have degree $2$ in $G$.

By Case 8,
$v_3$ has a neighbor $u_4$ distinct from $v_4$ and $u_1$,
and 
$v_5$ has a neighbor $u_5$ distinct from $v_6$ and $u_1$.
By Case 7, $u_4\not=u_5$.
If $u_4=u_3$,
then, by Lemma \ref{lemma0}(ii) 
for the local pair $v_4u_2$ and $v_3u_3$,
$u_5=u_2$, and $n(G)=9$.
Hence, $u_4\not=u_3$.
If $u_5=u_2$,
then, by Lemma \ref{lemma0}(ii) 
for the local pair $v_5u_2$ and $v_6u_3$,
$u_4=u_3$, and $n(G)=9$.
Hence, $u_5\not=u_2$.
By Lemma \ref{lemma0}(ii) 
for the local pair $v_4u_2$ and $v_3u_4$,
there is a $6$-cycle $v_4u_2v_7v_8u_4v_3v_4$,
where $v_7v_8 \in M$.
By Lemma \ref{lemma0}(ii) 
for the local pair $v_6u_3$ and $v_5u_5$,
there is a $6$-cycle $v_6u_3v_9v_{10}u_5v_5v_6$,
where $v_9v_{10}\in M$.
By Case 9,
all five edges 
$v_1v_2$,
$v_3v_4$,
$v_5v_6$,
$v_7v_8$, and 
$v_9v_{10}$ are distinct.
By Lemma \ref{lemma0}(ii),
$v_8$ and $v_{10}$ have degree $2$ in $G$.
By Case 8,
$v_7$ and $v_9$ have degree $3$ in $G$.
By Case 6,
$v_7$ is not adjacent to $u_4$,
and $v_9$ is not adjacent to $u_5$.
If $v_7$ is adjacent to $u_5$,
then, by Lemma \ref{lemma0}(ii) 
for the local pair $v_8u_4$ and $v_7u_5$,
$v_9$ is adjacent to $u_4$, and $n(G)=15$.
Similarly, 
if $v_9$ is adjacent to $u_4$,
then, by Lemma \ref{lemma0}(ii), 
$v_7$ is adjacent to $u_5$, and $n(G)=15$.
Hence, $v_7$ is not adjacent to $u_5$,
and $v_9$ is not adjacent to $u_4$.

We are again in a position to set up an inductive argument.
We give new names to some vertices,
which facilitates to recognize the underlying structure of some $L_k$.
Let 
$w_1'=v_6$,
$w_2'=v_{10}$,
$u_0'=v_2$,
$u_1'=u_3$, 
$u_2'=v_9$,
$v_0'=u_1$,
$v_1'=v_5$, and  
$v_2'=u_5$.
Let $k\geq 2$ be the largest positive integer such that 
\begin{itemize}
\item for every $i\in [k-1]$, 
\begin{itemize}
\item $u'_i$ has neighbors $u_{i-1}'$, $w_i'$, and $u_{i+1}'$,
\item $v'_i$ has neighbors $v_{i-1}'$, $w_i'$, and $v_{i+1}'$,
\item $w_i'$ has degree $2$ in $G$, 
\end{itemize}
\item for every $i\in [k]$, 
\begin{itemize}
\item if $i$ is even, then $u_i'w_i'\in M$,
\item if $i$ is odd, then $v_i'w_i'\in M$, and
\end{itemize}
\item the set
$$V=\{ v_1,v_2,v_3,v_4,v_7,v_8\}\cup \{ u_1,u_2,u_4\}\cup \bigcup_{i=1}^k \{ u_i',v_i',w_i'\}$$
contains $9+3k$ distinct vertices.
\end{itemize}
Note that all conditions are satisfied for $k=2$
by the previous discussion.
Let $x=u_k'$ and $y=v_k'$, if $k$ is even,
and let
$x=v_k'$ and $y=u_k'$, if $k$ is odd.
By Case 8, $x$ has degree $3$ in $G$.
If $x$ has a neighbor $z$ outside of $V$,
then, by Lemma \ref{lemma0}(ii)
for the local pair $xz$ and $w_k'y$,
$M$ contains an edge $v'v''$ with $v',v''\not\in V$
such that $v'$ is adjacent to $y$,
and $v''$ is adjacent to $z$.
By Lemma \ref{lemma0}(ii),
$v''$ has degree $2$ in $G$.
By Case 8,
$v'$ has degree $3$ in $G$.
Setting $w_{k+1}'$ equals to $v''$,
and setting $u'_{k+1}$ and $v'_{k+1}$ to $v'$ and $z$
suitably depending on the parity of $k$,
we obtain a contradiction to the maximality of $k$.
See Figure \ref{figcase10}.

\begin{figure}[H]
\centering\tiny
\begin{tikzpicture}[scale = 0.9] 
	    \node[label=below:\normalsize $u_1$] (u1) at (0,0) {};
	    \node[label=below:\normalsize $u_2$] (u2) at (1,0) {};
	    \node[label=below:\normalsize $u_3$] (u3) at (2,0) {};
	    \node[label=below:\normalsize $u_4$] (u4) at (-1,0) {};
	    \node[label=below:\normalsize $u_5$] (u5) at (3,0) {};

	    \node[label=above:\normalsize $v_1$] (v1) at (0,1) {};
	    \node[label=above:\normalsize $v_2$] (v2) at (1,1) {};
	    \node[label=above:\normalsize $v_3$] (v3) at (-1,1) {};
	    \node[label=above:\normalsize $v_4$] (v4) at (-2,1) {};
	    \node[label=above:\normalsize $v_5$] (v5) at (2,1) {};
	    \node[label=above:\normalsize $v_6$] (v6) at (3,1) {};
	    \node[label=above:\normalsize $v_7$] (v7) at (-3,1) {};
	    \node[label=above:\normalsize $v_8$] (v8) at (-4,1) {};
	    \node[label=above:\normalsize $v_9$] (v9) at (4,1) {};
	    \node[label=above:\normalsize $v_{10}$] (v10) at (5,1) {};
	    
	    \foreach \from/\to in {v1/v2, v3/v4, v5/v6, v7/v8, v9/v10, u1/v5, u1/v1, u1/v3, u2/v2, u2/v4, u2/v7, u3/v2, u3/v6, u3/v9, u4/v3, u4/v8, u5/v5, u5/v10}
	    \draw [-] (\from) -- (\to);
	    
	    \pgftext[x=-6cm,y=0cm] {\normalsize $G-V(M)$};
	    \pgftext[x=-6cm,y=1cm] {\normalsize $G(M)$};
	    
	    \draw[-,dotted] (5,0.5) -- (6,0.5);
	    
	    \node (u5) at (6.5,0) {};
	    
	    \node (w1) at (6,1) {};
	    \node (w2) at (7,1) {};
	    \node (w3) at (8,1) {};

	    \foreach \from/\to in {w1/w2, w3/w4, u5/w2, u5/w3, u6/w1, u6/w4}
	    \draw [-] (\from) -- (\to);
	    
	    \draw[-,dashed] (w3) -- (u4);
	    \draw[-,dashed] (u6) -- (v7);
	    
\end{tikzpicture}
\caption{The final situation in Case 10.}\label{figcase10}
\end{figure}
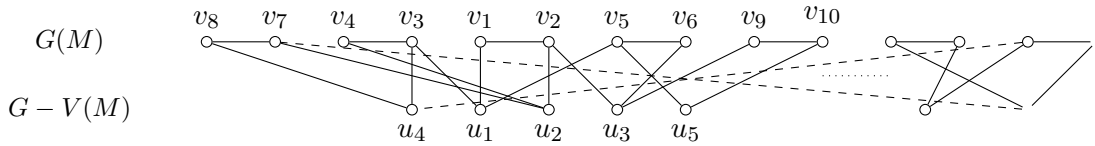
This implies that 
$x$ is adjacent to $u_4$
and 
that $y$ is adjacent to $v_7$,
which implies the $G$ is isomorphic to a graph in ${\cal B}_2$.
This completes the proof.
\end{proof}

\end{document}